\documentclass[12pt]{article}
\usepackage{amsmath,amsfonts,amssymb,amsthm,amscd}
\title{Model theory and groups}
\date{\today}
\author
{Anand Pillay\thanks{Supported by NSF grants  DMS 1665035 and DMS-1760212}\\University of Notre Dame  }

\newtheorem{Theorem}{Theorem}[section]
\newtheorem{Proposition}[Theorem]{Proposition}
\newtheorem{Definition}[Theorem]{Definition}
\newtheorem{Remark}[Theorem]{Remark}
\newtheorem{Lemma}[Theorem]{Lemma}

\newtheorem{Fact}[Theorem]{Fact}
\newtheorem{Example}[Theorem]{Example}
\newtheorem{Question}[Theorem]{Question}
\newtheorem{Problem}[Theorem]{Problem}
\newcommand{\R}{\mathbb R}   
\newcommand{\Q}{\mathbb Q}  
\newcommand{\Z}{\mathbb Z}  
\newcommand{\N}{\mathbb N}  
\newcommand{\F}{\mathbb F}
\newcommand{\C}{\mathbb C}

\begin{document}
\maketitle

\begin{abstract} This paper is about various ways in which groups arise or are of interest in model theory. In Section I briefly introduce three important classes of first order theories, stable theories, simple theories, and $NIP$ theories. Section 2 is about the classification of groups definable in specific theories or structures, mainly fields, and the relationship to algebraic groups.  In Section 3 I study generalized stability and definable groups in more detail, giving the theory of ``generic types" in the various contexts. I  also  discuss $1$-based theories and groups.  Section 4 is about the compact Hausdorff group $G/G^{00}$ attached to a definable group and how it may carry information in various contexts (including approximate subgroups).  In Section 5,  I discuss Galois theory, including the various Galois groups attached to first order theories, various kinds of strong types, and definable groups of automorphisms. In Section 6, I study various points of interaction between topological dynamics and definable groups,  in particular ``Newelski's conjecture" relating $G/G^{00}$ to the ``Ellis group".  And in Section 7, I touch on the model theory of the free group.

\end{abstract}

\section{Introduction and preliminaries}
I will discuss several points of interaction between model theory and group theory.  This is not so much about explicit applications of model theory to group theory, as about the model-theoretic perspective on groups, and some mathematical implications. It is a  personal account, influenced by my own interests,  preoccupations, and mathematical trajectory, and includes both established material, as well as some fairly recent developments, although there will be many important and  topical things  that I will say nothing about.

Model theory studies first order theories $T$, often complete. The study of specific first order theories, such as set theory, or differentially closed fields, can be identified with ``applications" of model theory, whereas the study of broad classes of first order theories (such as {\em all} theories, or stable theories)  is what is often considered as ``pure" model theory.  There are various invariants of a first order theory $T$. One is the category $Mod(T)$ of models of $T$ (where the morphisms are elementary embeddings). Most of the early material in a basic model theory course (Lowenheim-Skolem, prime models, saturated models, omitting types, Morley's theorem) concerns this category.  Another invariant is  $Def(T)$, the category of definable sets. $Def(T)$ can be thought of syntactically or semantically. Syntactically the objects are formulas $\phi({\bar x})$ (in the language of $T$, and where $\bar x$ denotes the tuple of free variables in $\phi$) up to equivalence modulo $T$. The morphisms are given by formulas $\psi({\bar x}, {\bar y})$ such that $T$ says that $\psi$ is the graph of a (partial) function. The semantic interpretation is the obvious thing and coincides with the syntactic interpretation when $T$ is complete. 

In many mathematical categories, such as the categories of algebraic varieties and differentiable manifolds, the group objects (algebraic groups, Lie groups, respectively) play an important role.  And the same holds for the category $Def(T)$. The group objects of $Def(T)$ are what we call the groups definable in $T$.  In many cases, especially on the stability side of model theory, definable groups arise naturally. For example if some definable relationship between definable sets $X$ and $Y$ needs parameters to be seen, then a nontrivial definable group appears (as a ``definable automorphism group"). The appearance of such groups was behind Hrushovski's positive answer to Shelah's question whether  ``unidimensional stable theories are superstable".  What we call {\em stable group theory} is the machinery of stability theory (independence, forking, ...) in the presence of a definable group action or definable group operation. It is considered part of ``general stability" and rather recently the fundamental theorem of stable group theory was seen to translate into a strong ``arithmetic regularity" theorem for finite groups $G$ equipped with a distinguished subset $A$, under a stability assumption on $A$ \cite{CPT1} (see also Section 4.5).

From the point of view of geometric stability theory and its generalizations, whether or not (infinite) groups are definable in a theory $T$ is one measure of the complexity of $T$. And assuming there are such infinite definable groups $G$, the complexity of the class of definable (in the ambient theory) subsets of $G$, $G\times G$, etc.  is another measure of complexity of $T$. 

There are  many things that I will not explicitly include in this paper. One is automorphism groups of $\omega$-categorical structures, which was an important subject in the 1980's, and has recently come back into prominence via Ramsey theory and topological dynamics. Another is the body of work around the Cherlin-Zilber conjecture  that simple groups of finite Morley rank are algebraic groups (over algebraically closed fields), which turned into a highly specialized subject with a tentative connection to the rest of stability theory. 
One more is the model theory of modules which again became a rather niche area. Groups of finite Morley rank (in the group language or with appropriate additional structure) and $R$-modules (in the usual language which has functions for scalar multiplication by elements of $R$) are examples of {\em stable groups}. 

Another non-included topic is Hrushovski's  group configuration theorem which gives an abstract model-theoretic context for recovering definable groups (see Chapter 5 of \cite{Pillay-GST}).  The proofs of function field Mordell-Lang in  \cite{Hrushovski-ML} are also an application of stable group theory, or rather its incarnation in specific stable theories.  The really new theorem is in positive characteristic which I will not discuss, although I will mention  the characteristic $0$ case in Section 2.4.

Pseudofinite groups (more or less ultraproducts of finite groups) will be touched on in a few ways, although not systematically. (But see Section 4.5.)  An early application of the model theory of groups definable in pseudofinite fields (and more generally groups definable in simple theories) to algebraic groups over finite fields (strong approximation) was in \cite{Hrushovski-Pillay-algebraicgroups-finitefields}  but will not be discussed in detail here. 

At the minimum I would hope that this article shows or reveals the flexibility of model theory in dealing with many different mathematical  topics and environments,  from a common point of view or standpoint.

%There are many contributors to the mathematics described in this article, who can be seen in the references and bibliography. Ehud Hrushovski and Saharon Shelah figure in a large way. But I also want to mention Daniel Lascar and Bruno Poizat;  for their approach to stability and stable groups, the model-theoretic interpretation of Galois theories,  and for pursuing the connections with algebraic geometry, algebraic groups, and differential fields, among many other things. 

I will assume a basic knowledge of model theory; languages $L$, formulas, first order theories $T$, structures, definable sets, types, elementary maps,  saturated models, as well as ``imaginaries".   See \cite{Marker} or \cite{Tent-Ziegler}, or my course notes \cite{Pillay-notes}.  Unless said otherwise variables $x,y$ range over finite tuples in the case of the traditional case of $1$-sorted structures, and over arbitrary sorts in the many sorted case (including $T^{eq}$). 

Notions such as quantifier-elimination and model-completeness are connected with the study of specific theories in a specific language or vocabulary.  $T$ has quantifier elimination if every formula $\phi(\bar x)$ is equivalent, modulo $T$ to a quantifier-free formula $\psi(\bar x)$.  And $T$ is model complete if whenever $M$ is a substructure of $N$, and both are models of $T$, then $M$ is an elementary substructure of $N$ ($M\prec N$).  Quantifier-elimination implies model completeness. One can always force quantifier elimination by the process of Morleyization: adjoining new relation symbols for all formulas (as well as their tautological definitions).  So from the point of view of model theory in and for itself, we can  always assume that the theory in question has quantifier elimination. But not for the study of specific theories in a specific language.  

Given a complete theory $T$ we typically work inside a sufficiently saturated model of $T$, which we call ${\bar M}$. Namely all models we consider are assumed to be elementary substructures of ${\bar M}$.  Sufficiently saturated means $\kappa$-saturated and strongly $\kappa$-homogeneous for some large $\kappa$.  Or if the reader doesn't mind assuming some set theory, just $\kappa$-saturated of cardnality $\kappa$. 
We will often  identify definable sets with the formulas defining them.  Definable sets are defined with parameters unless stated otherwise. We sometimes say $A$-definable to mean definable with parameters from a set $A$ (in the model ${\bar M}$.)  For a formula with parameters $b$ we sometimes write it as $\phi(x,b)$ to specify the parameters $b$ (where $\phi(x,y)$ is an $L$-formula, and $\phi(x,b)$ is the result of substituting constants  for the elements of $b$ for $y$).

Concerning algebraic geometry, we will typically take the naive point of view of algebraic varieties as point sets; sets of points of finite tuples in a field, defined by finite systems of equations.  The geometric point of view means considering such point sets in an ambient algebraically closed field.  In general an algebraic variety (in the sense of Weil for example) is obtained by glueing together finitely many affine varieties along suitable transition maps (analogous to how real topological or differential manifolds are built); the transition maps are isomorphisms in the sense of algebraic geometry between Zariski open subsets of the various charts.  This Weil point of view is thematically close to the model theoretic point of view of definable sets as  point sets in a saturated model, and is given a systematic treatment in \cite{Poizat-book}, repeated in \cite{Pillay-ACF}. 

Alternatively, a variety over a field $k$ means an integral separated scheme of finite type over $k$.

I would like to thank an anonymous referee for some crucial  mathematical corrections, as well as many typographical corrections. Thanks also to Daniel Hoffmann, Purbita Jana, and Chieu-Minh Tran for their corrections, comments, and suggestions. 

\subsection{Stability}  The complete theory $T$ is said to be {\em stable} if there do not exist a formula $\phi(x,y)$, a model $M$ of $T$, and tuples $a_{i},b_{i}$ from $M$ for $i<\omega$ such that $M\models \phi(a_{i},b_{j})$ iff $i\leq j$. 
By compactness this is equivalent to saying that for each $L$-formula $\phi(x,y)$ there is $k$ such that some (any) model $M$ of $T$ the formula $\phi(x,y)$ (i.e. the bipartitite graph defined by $\phi(M)$) omits the $k$-half graph. In the last case, we say that the formula $\phi(x,y)$ is $k$-stable.

\vspace{5mm}
\noindent
Stability and stable groups (groups definable in stable theories) will be discussed below, but in the meantime let us mention a certain strengthening
that will be relevant.   Following Morley, $T$ is said to be  total transcendental (t.t.), or in the case of a countable language,   $\omega$-stable, if  for any model $M$ of $T$, the Cantor-Bendixson rank on the Boolean algebra of formulas in any given finite number of variables, and with parameters in $M$,  is ``defined", namely ordinal valued.  When $M$ is an $\omega$-saturated model of $T$, this Cantor-Bendixson rank is precisely the Morley rank of the formula.

The $\omega$-stability of a complete theory $T$ (in a countable language) was related to uncountable ``categoricity" in Morley's work \cite{Morley}, where $T$ is said to be $\kappa$-categorical if any two models of $T$ of cardinality $\kappa$ are isomorphic.  This motivated (I guess) introduction of the general notion of stability, by Shelah, which played a central role in Shelah's program to classify theories according to whether or not there is some classification of their models by roughly speaking, cardinal invariants (\cite{Shelah}).  In the meantime stable theories have been understood as the ``logically perfect" first order theories. Moreover the  notion of stability of a formula (or bipartitite graph) is now seen as a pervasive notion in mathematics,  appearing in Grothendick's 1952 thesis, as well as subsequent work on topological dynamics, which subsumes stable group theory. 

Superstability is a property in between stability and $\omega$-stability. It is characterized by the existence of a certain ordinal-valued continuous rank on types and formulas:  $R(\theta(x)) \geq \alpha+1$ if there are unboundedly many complete types $p$ (say over the monster model) which contain $\theta$ and with $R(p)\geq\alpha$.  Countable, non superstable theories have $2^{\kappa}$ models of cardnality $\kappa$ for all $\kappa>\omega$, hence in the classification problem, $T$ can be assumed to be superstable. 

There are very few ``natural" stable theories and/or groups;  the theory of an infinite set with only equality, the theory of any algebraically closed field (in the ring language), the theory of any abelian group (in the group language), the theory of differentially closed fields, and recently the theory of the (noncommutative) free group.

\subsection{Simplicity}
The machinery used to understand models of stable theories and definable sets in stable theories is called ``stability theory". The invention of this machinery by Shelah represented a level of model-theoretic sophistication which was not present in other parts of model theory, and stability was considered as a kind of singularity. 
Actually from early on, Shelah was interested in unstable theories, and had defined in \cite{Shelah-simple} the notion of a {\em simple theory} as one where every complete type over a set $B$ does not divide over some ``small" subset $A$ of $B$.  Dividing and the related notion of forking give rise to the fundamental notion of (in)dependence in stable theories, ``$a$ is independent from $b$ over $C$",  satisfying a number of ``algebraic" properties, as well as a ``uniqueness" property.  In his 1996 thesis \cite{Kim-thesis}  Byungham Kim showed that all these  algebraic properties also hold for (non)dividing in the broader class of simple theories.  This showed that the underlying machinery of stability was not a singularity, and applied to contexts such as the asymptotic theory of finite fields, which up to that point had been studied using only elementary model-theoretic tools (but sophisticated algebraic-geometric tools).  Other examples of simple, unstable, theories are (the theory of) the random graph, as well as $ACFA$, the model companion of the theory of fields with a generic automorphism. Stable group theory extended naturally to groups definable in simple theories, with interesting modifications, which are reflected to some extent in the ``stabilizer theorem" for approximate groups (although the latter context is far from simple). 
The technical notions of dividing, forking etc. will be given later.

\subsection{NIP}
An opposite (from simplicity)  generalization of stability is in the direction of $NIP$ (not the independence property) theories. 
In the context of a complete theory $T$ in language $L$, an $L$-formula $\phi(x,y)$ has (or is) $NIP$ if there does not exist a model $M$ of $T$ and tuples $\{a_{i}:i\in \N\}$ and $\{b_{s}:s\in {\mathcal P}(\N)\}$ such that $M\models \phi(a_{i},b_{s})$ iff $i\in s$.  This can be finitized using compactness,  giving the notion of $k$-$NIP$,  and corresponds to the family of definable sets $\phi(x,b)$ as $b$ ranges over $M$ having finite Vapnik-Chervonenkis dimension.  As suggested by the last sentence, this notion has appeared independently in a number of mathematical contexts; model theory, learning theory, functional analysis.
The $NIP$ theories include the much-studied and well behaved theories with some ``topological" character, such as dense linear orderings, real closed fields, $p$-adically closed fields, various theories of Henselian valued fields, ordered abelian groups.  And stability implies $NIP$ (formula-by-formula).  In fact $T$ is stable iff $T$ is simple and $T$ is $NIP$. 

$o$-minimal theories are also examples of $NIP$ theories, and the general theory of $o$-minimality was (at least from the point of view of Steinhorn and me) meant to generalize aspects of stability theory, such as strong minimality, to an unstable context.  In \cite{Pillay-groups-o-minimal} I tried to generalize theorems about groups and fields definable in stable theories (more specifically in theories of finite Morley rank) to groups and fields definable in $o$-minimal theories.  This used notions of dimension analogous to Morley rank. However in subsequent work (since the early 2000's) the actual notions of stability theory  (forking, definability of types, finite satisfiability), rather than analogues,  were applied to $NIP$ theories and groups definable therein, with striking consequences.

This was another surprise for me, as I had believed that simple theories were the maximum class of  theories for which the technical notions of stability theory were meaningful.  Again Shelah was at the beginning of these developments.

\section{Groups definable in specific structures}
We survey here the attempts to describe or classify groups definable in particular theories or structures. Sometimes,  this has independent mathematical interest and/or applications.  Usually some kind of relative quantifier elimination is involved in the classification, and sometimes some general theory such as stability,  or simplicity, plays a nontrivial role.
Of course any group is a definable group in $ZFC$. This seems like a meaningless remark, but in later sections we may see that it is not so stupid.

Some of what we say in this section will depend on material appearing in later sections.   Algebraic groups are defined in 
the next (sub)section.   Algebraic groups (as well as their groups of $K$-rational points where $K$ is not necessarily algebraically closed) are considerd here as ``known", and a general theme will be the relationship of definable groups to algebraic groups in the various examples, and what is specifically new or interesting in the definable category.

\subsection{Algebraically closed fields} 
The important basic facts about the theory $ACF$ of algebraically closed fields, in the ring language ($+,\times, -, 0,1$), are that the completions are obtained by fixing the characteristic, that  $ACF$ has quantifier elimination (across the characteristics), and that each completion is strongly minimal (in the home sort), namely all definable (with parameters) subsets of the home sort are finite or cofinite.  In particular each completion of $ACF$ is $\omega$-stable.  See \cite{Poizat-book} for more details about this section. 

Let us now fix the characteristic to be $0$, so we are dealing with the complete theory $ACF_{0}$. 
We let $K$ denote a model, namely an algebraically closed field of characteristic $0$ such as the complex numbers. 
One of the basic problems of algebraic geometry is the classification up to birational isomorphism of irreducible algebraic varieties over $K$. There is  a bijection between irreducible varieties $V$ over $K$, and complete types $p(\bar x)$ over $K$, ($p$ being the generic type over $K$ of $V$). And  $V, W$ are birationally isomorphic iff there are realizations ${\bar a}$ of $p_{V}$ and ${\bar b}$ of $p_{W}$ such that ${\bar a}$ and ${\bar b}$ are interdefinable over $K$ if and only there are definable subsets $V_{1}$, $W_{1}$ of $V, W$ respectively of maximal Morley rank and a definable bijection between $V_{1}$ and $W_{1}$.  So in this sense the birational classification of algebraic varieties coincides with the classification of definable sets up to definable bijection with respect to $ACF_{0}$.  However this observation  does not seem to have contributed much to the birational classification problem. 

An algebraic group is a group object in the category of algebraic varieties.  Weil's theorem that algebraic groups can be recovered from birational data, translates (with some work, prived by Hrushovski) into the theorem that definable groups coincide with algebraic groups. Namely any definable group can be definably equipped with the structure of an algebraic group, and that definable isomorphisms between two definable groups coincide with isomorphisms (in the sense of algebraic groups) between the corresponding algebraic groups.  This is somewhat subtle and it is worth paying attention to the precise definitions and notions. In fact it was via this theorem that I got an inkling of what an abstract algebraic variety is. 

On the other hand there is a structure theory for algebraic groups. There are two extreme cases of algebraic groups (over an algebraically closed field $K$),  namely linear algebraic groups (algebraic subgroups of some $GL(n,K)$) and abelian varieties  (algebraic groups whose underlying variety is a projective variety). An arbitrary (irreducible, or connected) algebraic group is an extension of an abelian variety by a linear algebraic group.  These two kinds of algebraic groups seem, from the outside, to belong to different parts of mathematics, in the sense that there are very few theorems dealing with {\em all} algebraic groups (other than theorems about commutative algebraic groups).  In fact the theory of commutative algebraic groups is deeply connected with the theory of differential forms of the first, second, and third, kind, on algebraic curves, as well as ``arithmetic algebraic geometry". 

In terms of identifying the definable groups in the theory $ACF$ there is nothing more to be said.
One can ask what exactly model theory can contribute to algebraic geometry via the first order theory $ACF$. Probably not so much, other than the group configuration theorem. However, aspects of algebraic geometry can be captured in richer ``tame" theories, as will be discussed below.

\subsection{Real closed fields}
We are dealing here with the theory $RCF$ of the structure $\R$ in the ring language mentioned above, or its definitional expansion $RCOF$ in the ordered ring language (noting that the nonnegative elements are precisely the squares).  The standard model is the field of real numbers with the usual ordering.  We will focus here on groups definable in the standard model, although things generalize suitably to arbitrary models, i.e. real closed (ordered) fields. Tarski's quantifier elimination theorem says that the sets definable in the structure $(\R,+,\times)$ are  the {\em semialgebraic sets}; where a semialgebraic set is, by definition,  a subset of $\R^{n}$ defined by  a finite disjunction of sets defined by $f({\bar x) = 0 \wedge  g_{1}(\bar x}) > 0 \wedge ... \wedge g_{k}(\bar x) > 0$, for $f,g_{i}$ polynomials over $\R$. 

We call groups definable in $(\R,+,\times)$  {\em semialgebraic} groups (that is, both the universe of the group and the graph of the group operation are semialgebraic sets). On the face of it there is no continuity in the definition of semialgebraic.  I will give a rather extended discussion around the problem of classifying semialgebraic groups. 

The analogue of the theorem that definable groups in $ACF$ can be definably equipped with the structure of algebraic groups, is that a semialgebraic group can be definably equipped with the structure of a {\em Nash group} and that moreover Nash groups are semialgebraic. (And again  there will be an equivalence of categories between semialgebraic groups and Nash groups). 

Nash manifolds originated in  work of John Nash \cite{Nash}, where he introduced the notion of a {\em real algebraic manifold}, which in current parlance is a compact affine Nash manifold. They appear again in \cite{Artin-Mazur} and were systematically studied in  \cite{Shiota}.  We follow \cite{Shiota}, in particular making the distinction between Nash manifolds and locally Nash manifolds. The definition is as follows: a Nash manifold is a real analytic manifold with a covering by finitely many open sets each of which is diffeomorphic to some open semialgebraic subset of some $\R^{n}$, and such that the transition functions are Nash, namely both analytic and semialgebraic.  So a Nash manifold  can be described as a ``semialgebraic real analytic manifold". By a Nash group we mean a group object in the category of Nash manifolds, so the underlying set is a Nash manifold and the group operation is Nash, namely analytic and semialgebraic when read in the semialgebraic charts. A Nash group is also a Lie group. The category of Nash groups is in between that of real algebraic groups and that of Lie groups  (in the sense of inclusions of categories). 

A more general theorem is that a group definable in an $o$-minimal expansion of the real field has definably the structure of a Lie group (in particular a topological group), and is proved in \cite{Pillay-groups-o-minimal} by adapting the ideas in the $ACF$ case.  This theorem specializes to the statement that semialgebraic groups ``equal"  Nash groups in the case of $(\R,+,\times)$.

What we call a {\em real algebraic group}  can be best described as the group  $G(\R)$ of real points of an algebraic group defined over $\R$.   Such a real algebraic group $G(\R)$  is also a topological (in fact Lie) group and has a connected component $G(\R)^{0}$ which is semialgebraic and of finite index. For example when $G = GL_n$, then $G(\R)^{0}$ is the collection of matrices of positive determinant.  Any open subgroup of $G(\R)$ is semialgebraic and lies in between $G(\R)$ and $G(\R)^{0}$.  With a slight abuse of terminology we will extend the notion of real algebraic group to include their open subgroups, in particular their topological connected component.  Notice that with this terminology, any semialgebraic subgroup of a real algebraic group will be real algebraic, as it will be open in its Zariski closure.

In any case, we will take real algebraic groups to be something known. And what remains to be done is to give a classification of Nash groups in terms of  real algebraic groups, their covers, and their quotients.

Let us be more precise. 

In \cite{Hrushovski-Pillay-groupsinlocalfields} it was shown that any Nash group is definably locally isomorphic to a real algebraic group, in the sense that there are open definable  neighbourhoods of the  identity of the Nash group and some real algebraic group and a definable bijection between them which respects the group operations (whenever they are defined). 
It follows that given any connected Nash group $G$ there is a (connected) real algebraic group $H$ and a locally Nash isomorphism between the universal covers $\tilde G$, $\tilde H$ of $G, H$, where these universal covers are both {\em locally Nash groups}. Locally Nash just allows an infinite rather than finite covering by open semialgebraic sets.   Hence $G$ itself is a quotient of $\tilde H$ by some discrete subgroup, and the question is to classify the discrete subgroups of $\tilde H$ such that the quotient, which is a priori locally Nash has a compatible Nash structure. 

This was basically solved in the commutative case (by me and Starchenko), a few years ago, but not yet written up. 

On the other hand, imposing the condition of being {\em affine} is a very strong condition on Nash groups. Affineness of a Nash manifold $X$ means that there is a Nash embedding  of $X$ into 
some $\R^{n}$.   Any real quasiprojective variety is affine, due to sterographic projection. Hence all real algebraic groups are affine. 
It was ``proved" in \cite{Hrushovski-Pillay-groupsinlocalfields} with a corrected proof in \cite{Hrushovski-Pillay-affine} that an affine Nash group $G$ (say connected) is virtually algebraic, namely there is a connected real algebraic group $H$ and a definable surjective homomorphism from $G$ to  $H$, with finite kernel. Namely affine Nash groups are finite covers of real algebraic groups.   The category of finite covers of real algebraic groups goes outside the real algebraic category. As is well-known there are finite covers (as Lie groups) of $SL(2,\R)$ which have no linear representation. 

On the other hand there do exist non-affine Nash groups, the simplest being $\R/\Z$ viewed as a Nash group.  One should be careful here: the category of Nash groups is in between the category of real algebraic groups and real Lie groups. 
All $1$-dimensional compact, connected Nash groups are isomorphic as Lie groups but not as Nash groups. These include $SO_{2}(\R)$, $E(\R)^{0}$ (for elliptic curves $E$ over $\R$, which are all real algebraic, hence affine as Nash groups), and $\R/\Z$ as above (the real unit interval with addition modulo $1$ and with $0$ and $1$ identified) and other examples, which are all isomorphic as Lie groups, but not as Nash groups.

\subsection{The $p$-adics}
 I will focus here on analogous questions for the field $\Q_{p}$ of $p$-adic numbers to those for the reals discussed in the last section.  In spite of the model theory of valued fields being a major area of research, these analogies do not seem to have been explored so much, maybe because the expected answers will be not so interesting.
I do not want to go into details of the model theory of the $p$-adics, other than a brief survey.  Macintyre proved a quantifier elimination theorem for $Th(\Q_{p},+,\times,-,0,1)$ after adding predicates for the $nth$ powers, for all $n$ \cite{Macintyre}.  This is in analogy with Tarski's  quantifier elimination for the (theory of the) real field after adding a predicate for the squares. 
The topology on $\Q_{p}$ is the valuation topology: where the basic open neighbourhoods of a point $a$ are $\{x:v(x-a) \geq r\}$, for $r\in \Z$.  With this topology $\Q_{p}$ is totally disconnected, and the ring $\Z_{p}$ of $p$-adic integers (valuation $\geq 0$) is compact, so profinite.  We have the notion of $p$-adic analytic function on an open subset of $\Q_{p}^{n}$, given by locally convergent power series, and hence we obtain the notion of a $p$-adic analytic manifold, and thus also a $p$-adic analytic group.

In the context of $\Q_p$, ``semialgebraic" means definable in the structure $(\Q_{p},+,\times,-,0,1)$ (which can be made more explicit using Macintyre's quantifier elimination theorem), and we obtain the notion of a $p$-adic Nash group, in analogy with the real case.  And again we have an equivalence of categories between groups definable in the field $\Q_{p}$ and $p$-adic Nash groups. 

As in the real case, by a $p$-adic algebraic group we first mean something of the form $G(\Q_{p})$ where $G$ is an algebraic group defined over $\Q_{p}$.  Such a group $G(\Q_{p})$ is of course a $p$-adic Nash group. And as in the real case, a problem or question is to find the relationship between $p$-adic Nash groups and $p$-adic algebraic groups. 

In the real case, passing from real algebraic groups to their open subgroups was a mild move; there is a smallest one (the topological connected  component) which has finite index and is semialgebraic and so all open subgroups are semialgebraic.
What happens in the $p$-adic case?  Note that $\Z_{p}$ is an (open) semialgebraic subgroup of the $p$-adic algebraic group $\Q_{p}$  (additively for example), but with infinite index.  Moreover any semialgebraic subgroup of a $p$-adic algebraic group is open in its Zariski closure.   So in analogy with the real case above  it makes sense (to me) to make a couple of moves:
first, as was  stated informally in \cite{Pillay-application-real-p-adic}.
\begin{Problem} Is any open subgroup of a $p$-adic algebraic group,  $p$-adic semialgebraic?
\end{Problem}

And secondly to redefine a $p$-adic algebraic group as a semialgebraic subgroup of a group of the form $G(\Q_{p})$ where $G$ is an algebraic group over $\Q_{p}$.

With this new definition, clearly all $p$-adic algebraic groups are $p$-adic Nash, and we believe that, in contradistinction with the real case, the converse holds too.

\begin{Problem} Up to definable (or semialgebraic or $p$-adic Nash) isomorphism, $p$-adic Nash groups are precisely the $p$-adic algebraic groups.
\end{Problem} 

\vspace{5mm}
\noindent
Both $RCF$ and the theory of $\Q_p$ are $NIP$ but this does not play much of a role in the classification results above. It was rather a relative quantifier elimination, possibly cell-decomposition, and the fact that model-theoretic and field-theoretic algebraic closures coincide.  In fact the latter, together with a delicate application of the group configuration theorem inside an ambient {\em algebraically closed} field, yields a local (in the topological sense) semialgebraic isomomorphism between a given definable group $G$ and a real/$p$-adic algebraic group, which is the first step in the clsssification results and problems.

We have been discussing definable, rather than interpretable groups (which coincide for $RCF$), and the classification of the interpretable groups in $\Q_{p}$ must take account of groups living in the value group $(\Z,+,<)$ as well as the interaction with $p$-adic Nash groups.

There have been major developments in the model theory of valued fields, where the theory $ACVF$ of algebraically closed valued fields provides an appropriate ``universal domain".  Groups definable therein have been studied too, and this is rather different, in spirit, to what is discussed above (although related).

%\subsection{Pseudofinite fields} 
%A pseudofinite field is an infinite model of the theory of finite fields (in the ring language).  The completions of the theory of pseudofinite fields are determined by the isomorphism types of the relative algebraic closures of the prime fields. 
%After adding some constants and suitable axioms for them, the theory is model complete, but in a strong sense which implies that model theoretic and (relative) field theoretic algebraic closures coincide.  These completions are also simple. In fact the development of ``simple theories" went hand-in-hand with the analysis of concrete theories and structures such as pseudofinite fields, where Hrushovski had developed already adaptations of stable group theory to a certain context of ``$S_{1}$-theories" (subsequently put in the context of simple theories). 
%To cut a long story short, ... Do basic results and mention applications.   Galois cohomology.... too.  Strong approximation. 

\subsection{Differentially closed fields}
The theory of differentially closed fields of characteristic $0$, $DCF_{0}$  (in the language of unitary rings together with a symbol $\partial$ for the derivation)  has been among the richest stable theories for many years and is still being investigated. Some basic references here are \cite{Marker-differential} and \cite{Pillay-foundational}. This section is actually fairly close to \cite{Marker-differential}.

$DCF_{0}$ has quantifer elimination (in the language given above), is $\omega$-stable and has  infinite Morley rank: the formula ``x=x" has Morley rank $\omega$.  The connection between definability in $DCF_{0}$ and actual differential equations and their solutions is subtle and worth exploring. At the least there is some impact on  {\em functional transcendence}, the investigation of algebraic relations between solutions of an algebraic differential equation equation over $\C(t)$ for example. 

Often ${\mathcal U}$ is used to denote a saturated model of $DCF_{0}$, or a universal domain for differential algebra, in the sense of Kolchin. An important feature (shared by $ACF$ and $RCF$) is that $DCF_{0}$ has elimination of imaginaries, meaning that interpretability and definability coincide.

${\mathcal U}$ as well as its field ${\mathcal C}$ of constants ($\partial (x) = 0$) are algebraically closed fields.  So we can view both $\mathcal U$ and $\mathcal C$ as ``universal domains" for algebraic geometry.

A differential polynomial in differential indeterminates $x_{1},..,x_{n}$ over a differential field $(K,\partial)$ is a polynomial over $K$ in indeterminates 
\newline
$x_{1},..,x_{n}, \partial(x_{1}),...,\partial(x_{n}), \partial^{(2)}(x_{1}),...,\partial^{(2)}(x_{n}),..........$.  Such a differential polynomial can be evaluated at a point $(a_{1},..,a_{n})$ in any differential field extending $(K,\partial)$. 

The analogue of an affine algebraic variety is an affine differential algebraic variety, a subset of ${\mathcal U}^{n}$ defined as the common $0$-set of  a finite system of differential polynomials $P(x_{1},..,x_{n})$ with coefficients from $\mathcal U$.   When $\partial$ does not appear in the differential polynomials $P$ then this is just an algebraic variety. 
The so-called Kolchin topology on an affine differential algebraic variety $X$ has as closed subsets of $X$ the differential algebraic subvarieties. 

There have been various attempts to develop ``differential algebraic geometry" in analogy with algebraic geometry. 
(By the way what we call differential algebraic geometry is something distinct from differential geometry.) This includes a Weil-style approach, where abstract differential algebraic varieties are built from affine differential algebraic varieties by gluing where the transition maps are ``differential rational" isomorphisms between Kolchin opens  (for example). There are also scheme-theoretic approaches.  Kolchin's own definitions (even of algebraic groups) are somewhat idiosynchratic, influenced by his account of differential Galois theory.  Whatever the relevant definitions, the categories of definable groups and differential algebraic groups coincide, as in $ACF$ for definable vs algebraic, and as in $RCF$ for definable vs Nash.
But more important is the simplifying fact (theorem) that up to definable isomorphism any definable group is a definable (differential algebraic) subgroup of an algebraic group.  However in contradistinction with the cases of $RCF$, $p$-adics, pseudofinite fields, there are qualitatively new phenomena appearing in such definable subgroups.  One reason for this is that an algebraic group $G$ (as a definable group in ${\mathcal U}$)  has infinite Morley rank, whereas there will always be definable subgroups of finite Morley rank. In fact the finite Morley rank definable groups are part of the rather rich theory of definable sets of finite Morley rank (in $DCF_{0}$). 

In certain technical senses the building blocks of finite Morley rank definable sets (in a given, saturated, model $\bar M$ of a theory $T$) are the {\em strongly minimal sets}. Recall that a strongly minimal set $X$ is a definable set in ${\bar M}$
which is infinite and such that moreover any definable subset of $X$ is finite or cofinite.  The identification of strongly minimal sets and their properties in $DCF_{0}$ is an open problem, with many ramifications.
The case of strongly minimal definable groups, on the other hand, is well-understood and is the key input to a model-theoretic proof of function field Mordell-Lang in characteristic $0$.  Theorems 2.3 and 2.4 below are due to Hrushovski and Sokolovic (unpublished), together with the structure of $1$-based groups (which will be discussed later). But a self contained account is in \cite{Nagloo-Pillay}. 

\begin{Theorem} Let $G$ be a strongly minimal definable group in ${\mathcal U}$. Then up to definable isomorphism $G$ is one of the following:
\newline
(i) $({\mathcal C},+)$,
\newline
(ii)  $({\mathcal C}^{*},\times)$.
\newline
(iiii) $E({\mathcal C})$ where $E$ is an elliptic curve (one-dimensional abelian variety) defined over ${\mathcal C}$.
\newline
(iv) $A^{\sharp}$ where $A$ is a simple abelian variety defined over ${\mathcal U}$ which is not isomorphic (as an algebraic group) to an abelian variety defined over $\mathcal C$, and where $A^{\sharp}$ denotes the Kolchin closure of the torsion subgroup of $A$. 
\end{Theorem}

Some explanations are in order. Abelian varieties were defined above in 2.1 as (infinite) connected algebraic groups with underlying variety projective.  So as such these are definable groups in ${\mathcal U}$, but of infinite Morley rank.  $A$ is said to be simple if it has no proper nontrivial abelian subvarieties (algebraic subgroups). When $A$ is of dimension one (an elliptic curve) $A$ is simple, but there examples in all dimensions $>1$.  The group of torsion points of an abelian variety is Zariski-dense, and its Kolchin closure (the smallest definable subgroup of $A$ containing the torsion) $A^{\sharp}$  is a connected definable subgroup of $A({\mathcal U})$ of finite Morley  rank. 

We also have:
\begin{Theorem}
In Case (iv) of Theorem 2.3, $A^{\sharp}$ is ``modular" or ``$1$-based", meaning that or at least equivalent to the property that every definable (in ${\mathcal U}$) subset of any Cartesian product $A^{\sharp} \times ....\times A^{\sharp}$ is a (finite) Boolean combination of cosets (translates) of definable subgroups, where moreover these definable subgroups are defined over the algebraic closure of the set of parameters defining $A^{\sharp}$. 
\end{Theorem}

Mordell-Lang and Manin-Mumford type problems originate with the Mordell conjecture (proved by Faltings) that a curve of genus $\geq 2$ over the rationals has only finitely many rational points.

Hrushovski proved function field Mordell-Lang in all characteristics \cite{Hrushovski-ML}. 

A very special case of function field Mordell-Lang as well as function field Manin-Mumford follows  quickly from Theorems 2.3 and 2.4.

\begin{Theorem} Let $k<K$ be algebraically closed fields of characteristic $0$, and $A$ a simple abelian variety defined over $K$ which has $k$-trace $0$ (namely is not isomorphic to an abelian variety defined over $k$).  Let $X$ be an irreducible subvariety of the Cartesian product $A^{n}$ (defined over $K$) such that the intersection of $X$ with the group of torsion points of $A^{n}$ is Zariski-dense (in $X$).  Then $X$ is a translate of an abelian subvariety of $A^{n}$

\end{Theorem}
\begin{proof}  After extending $K$ equip $K$ with the structure of a differentiall closed field $(K,\partial)$ whose field of constants is $k$.  Then by  assumption and 2.4, $A^{\sharp}$ is modular. Hence $X\cap (A^{\sharp})^{n}$ is a
 finite union of cosets of subgroups. (A little argument is needed here to show that a Kolchin closed set which is a Boolean combination of cosets of definable subgroups is actually a finite union of such cosets.)
As $(A^{\sharp})^{n}$ contains the torsion of $A^{n}$, by our assumptions $X\cap (A^{\sharp})^{n}$ is Zariski-dense in $X$ from which it follows that $X$ is itself a translate of a subgroup (abelian subvariety) of $A^{n}$. (Again a little argument is required to show that the Zariski closure of a finite union of cosets is a finite union of cosets of algebraic subgroups, and in the case or irreducibility is a single such coset.) 

\end{proof}
Of course it takes considerable insight to see the connection between  Mordell-Lang type problems and the model theory of differentially closed fields, although Buium \cite{Buium} had already made the connection with differential algebra.
Nevertheless it is simply the fact that  model-theoretic and definability considerations are even {\em relevant} to or meaningful for algebraic and diophantine geometric problems that is striking for me, and is an instantiation of the ``unity of mathematics".  Hopefully other parts of this paper will serve a similar purpose.

\section{Stability, stable group theory and generalizations}
The general theory of stable groups is due to Poizat (building on Cherlin, Shelah, Zilber in special cases). 
It appears in \cite{Poizat-book} but the original paper is \cite{Poizat-generics}.

Stable group theory is stability theory in the presence of a transitive definable group action, a special case of which is simply a definable group.  A very special case is when the ambient theory is $ACF$ (or rather one of its completions, $ACF_{0}$ or $ACF_{p}$).  
In this case the definable groups are algebraic groups as mentioned in Section 2.1. And much of the language of stable group theory, such as connected components, is borrowed from algebraic group theory. 
Also there is a rather important aspect of stability theory called ``local stability" where an assumption of stability is made only on  a single formula $\phi(x,y)$ rather than all formulas.  This also has a group version, local stable group theory, appearing in \cite{Hrushovski-Pillay-groupsinlocalfields}. 

This is the appropriate point to introduce, or recall, the key notions of dividing and forking (by Shelah). Let $T$ be an arbitrary complete theory for now, and work in a saturated model $\bar M$ of $T$. We stick with the conventions in the introduction. 
A formula $\phi(x,b)$ is said to {\em divide} over a set $A$ of parameters, if there is some $A$-indiscernible sequence $(b_{i}:i<\omega)$ with $b_{0} = b$ such that  $\{\phi(x,b_{i}): i <\omega\}$ is inconsistent, equivalently $k$-inconsistent for some $k<\omega$.  The general idea is that this notion of dividing captures model-theoretically  notions of ``dependence" from both field theory and linear algebra.  The notion of forking was introduced to show that the corresponding notion of {\em independence} has a kind of ``built-in" extension property: a formula $\phi(x,b)$ is said to {\em fork} over $A$ if $\phi(x,b)$ implies a finite disjunction of formulas each of which divides over $A$.  What is quite amazing is how many phenomena are captured by this apparently obscure notion. 

Anyway we will say that a tuple $a$ is $f$-independent (where $f$ is for forking) from a set $B$ of parameters over a subset $A$ of $B$, if  {\em no} formula $\phi(x,b)$ in the type of $a$ over $B$ forks over $A$.  

When $T$ is stable, this independence notion has some properties which have an ``algebraic" character, together with a certain {\em uniqueness} property.  The algebraic properties include things such as symmetry ($a$ is $f$-independent from $B$ over $A$ iff for each tuple $b$ from $B$, $b$ is $f$- independent from $A\cup\{a\}$ over $A$), and extension, or existence  (for any $a$, and $A\subseteq B$, there is $a'$ with the same type as $a$ over $A$ such that $a$ is $f$-independent from $B$ over $A$). 
For some time, many model-theorists, at least I, had the impression that it was these algebraic properties of $f$-independence in stable theories that were crucial for, or even characteristic of, stability. But it is now clear that it is the uniqueness aspect which is characteristic. This uniqueness property says that working over a base set $A$ which is algebraically closed in the sense of ${\bar M}^{eq}$ (for example if $A$ is an elementary substructure of ${\bar M}$),
given $a$ and $A\subseteq B$, there is a {\em unique} type of $a'$ over $B$ such that $a$ and $a'$ have the same type over $A$ and $a'$ is $f$-independent from $B$ over $A$.   Another way of expressing this is if $p(x)$ is a complete type over $A$, and $I_{p}$ is the collection of formulas $\phi(x,b)$ with arbitrary parameters from ${\bar M}$ such that $p(x)\cup\{\phi(x,b)\}$ forks over $A$, then $p(x)\cup\{\neg\phi(x,b): \phi(x,b)\in I_{p}\}$ determines or axiomatizes a unique complete type $p'(x)$ over ${\bar M}$.  This is a vast stability-theoretic generalization of the fact that if $K$ is an algebraically closed field and $V$ is an algebraic variety over $K$ which is $K$-irreducible, then $V$ is absolutely irreducible. 

When the base set $A$ of parameters is a model (elementary substructure of ${\bar M}$), we have  the following:
\begin{Fact} Suppose $T$ is stable and $M$ is a model and $B\supseteq M$ and $p(x)\in S_{x}(B)$. Then the following are equivalent:
\newline
(i) $p$ does not fork over $M$,
\newline
(ii) $p$ is definable over $M$,
\newline
(iii) $p$ is finitely satisfiable in $M$.
\newline
Moreover $p$ is the unique ``nonforking extension" of its restriction $p|M$ to $M$. 
\end{Fact}

Let us give some explanations of (ii) and (iii). The notion of a definable type is central to model theory, and may have originated in work of Gaifman on models of Peano arithmetic.  Anyway given a complete type $p(x)$ over a set $B$ and a subset $A$ of $B$, $p$ is said to be definable (over $A$) if for each formula $\phi(x,y)$ of the underlying language, there is a formula $\psi_{\phi}(y)$ over $A$ (i.e. with parameters from $A$) such that for each tuple $b$ from B, $\phi(x,b)\in p$ iff ${\bar M}\models \psi_{\phi}(b)$.  In the case, as in  Fact 3.1 above, that $A$ is a model $M$ (elementary substructure of $\bar M$), then the ``defining schema"  which takes $\phi$ to $\psi_{\phi}$ yields a complete type over any set $B$ extending $M$, which coincides with the unique nonforking extension of $p$ over $B$.  On the other hand we say that $p\in S_{x}(B)$ is finitely satisfiable in $M$ if every formula $\phi(x)$ in $p$ (which will of course have parameters in $B$, not necessarily in $M$) has a solution, or realization, in $M$. 

\subsection{Stable groups}
Stable group theory is  an adaptation of the above machinery to the context of definable homogeneous spaces (a transitive definable action of a definable group $G$ on a definable set $X$), although we will focus here on the case where $X = G$.
As mentioned in several places earlier, by a stable group we usually mean a group $G$ definable in a model of stable theory. To understand the model theory we often assume that the ambient model is saturated. It is basically equivalent to think of a stable group, as a $1$-sorted $L$-structure $M$ (some language $L$) such that the basic or home sort is a group, where the group operation is the interpretation of a symbol in the language $L$, and the theory of $M$ is stable.
And sometimes one considers the language $L$ to consist only of the group operation (and maybe inversion, and the identity element).  One can also think of a stable group syntactically (relative to a given theory) namely as a pair of formulas, one defining the universe of the group, and the other defining the group operation. 

There are several  aspects to stable group theory; including, chain conditions, stabilizers, generic types, and connected components.

\begin{Proposition} Let $G$ be a stable group. Let $\phi(x,y)$ be an $L$-formula where $x$ ranges over $G$, and such that for each $b$, $\phi(x,b)$ defines a subgroup of $G$. Then (working in a saturated model) any intersection of subgroups of $G$ defined by instances $\phi(x,b)$ of $\phi(x,y)$ is a finite sub-intersection.
\end{Proposition}

When $G$ is $t.t$ in the sense that the Morley rank of the formula defining $G$ is an ordinal (rather than $\infty$), then we have the DCC on the family of {\em all} definable subgroups of $G$.

When $G$ is superstable we have a weaker descending chain condition: there is no infinite chain $G_{1}\supset G_{2} \supset ... $ of definable subgroups such that $G_{i+1}$ has infinite index in $G_{i}$. 

Given a stable group $G$ and model $M$, $G(M)$ acts on the space $S_{G}(M)$ of complete types over $M$ containing the formula defining $G$.  The action is:  $gp = tp(ga/M)$ where $a$ realizes $p$. 

It follows from definability of types and Proposition 3.2 that:

\begin{Proposition} Let $G$ be a stable group, $M$ a model, and $p(x)\in S_{G}(M)$  (complete types over $M$ concentrating on $G$), then $Stab(p) =_{def} \{g\in G(M): gp = p\}$  is an intersection at most $|T|$ many   definable subgroups of $G(M)$.

When $G$ is $t.t$, $Stab(p)$ is outright definable. 
\end{Proposition} 

In the $t.t$ case (for example when $T$ is $ACF_{0}$ and $G$ is just an algebraic group), the above results generalize the DCC on algebraic subgroups of $G$ and the fact that for any subvariety $X$ of $G$, $Stab(X) = \{g\in G: g(X) = X\}$ is an algebraic subgroup of $G$.    Another notion from algebraic group theory is ``connected component" (of the identity): given an algebraic group $G$, one can write $G$ as a disjoint finite union of connected (in fact irreducible) sets, each clopen in the Zariski topology, and $G^{0}$ denotes that connected component which contains the identity. It is a normal subgroup of $G$ of finite index (and the other components are just translates of $G^{0}$). 
This extends to the general case when $G$ is $t.t$ (although there is no Zariski topology on $G$): There is a smallest definable subgroup of $G$ of finite index, which we call again $G^{0}$, the connected component of $G$ (strictly speaking of the identity in $G$).

In the general stable case, one should be more careful. For a given formula $\phi(x,y)$ all of whose instances define subgroups, by Proposition 3.2 there is a smallest subgroup of $G$ of finite index and which is definable by a conjunction of instances of $\phi(x,y)$. We can call this $G_{\phi}^{0}$ (if there is no ambiguity). And by $G^{0}$ we mean the intersection of all the $G_{\phi}^{0}$, as $\phi$ varies.  $G^{0}$ is a normal subgroup of $G$.  The $G_{\phi}^{0}$ are all $\emptyset$-definable, so their intersection is $\emptyset$-type-definable. 
We will in general only ``see" $G^{0}$ and the quotient $G/G^{0}$ in a saturated model, and this quotient naturally has the structure of a {\em profinite group} (inverse limit of finite groups). 

\begin{Example} When $G$ is $(\Z,+)$  (as both a group and a structure), then $G^{0}$ is the intersection of the subgroups  $n\Z$, which we only really see in a saturated elementary extension.
\end{Example}

So (working in a saturated model) $G/G^{0}$ is the inverse limit of the $\Z/n\Z$, namely $\hat \Z$ the so-called profinite completion of $\Z$.  $(\Z,+)$ is the canonical superstable, non totally transcendental (t.t.)  group. 

We now discuss what I like to call the fundamental theorem of stable group theory  (FTSGT), which combines stabilizers, connected components, and  ``genericity".  A natural notion of a ``large" set $X$  in a group $G$ is one such that finitely many (left) translates of $X$ cover $G$.  In the topological dynamics literature this kind of set is called syndetic.
When $G$ is a definable group and $X$ a definable subset we (in model theory) call $X$ (left) generic. 
Fix now a stable group and a model $M$. 
\begin{Lemma} The collection of definable over $M$ subsets of $G$ which are non (left) generic, is  a proper ideal in the Boolean algebra of $M$-definable subsets of $G$.
\end{Lemma}

Hence there are types $p\in S_{G}(M)$ which avoid this ideal, namely such that all formulas in $p$ are generic, and we call such $p$ a (left) generic type of $G$ over $M$.  Notice that the collection of generic types is closed in $S_{G}(M)$ (by its definition).  Now any $p(x)\in S_{G}(M)$ determines a coset of $G^{0}$ (as it determines a coset of each $G_{\phi}^{0}$). So we have a natural map $\iota$ from the collection $Gen$ of generic types to $G/G^{0}$.

\begin{Theorem} (with above notation and assuming stability).
(i) $\iota$ determines a homeomorphism between $Gen$ and (the profinite group) $G/G^{0}$.
\newline
(ii) $G(M)$ acts transitively on $Gen$ and for each $p(x)\in Gen$, $Stab(p) = G^{0}$.
\newline
(iii) Left and right genericity coincide.
\end{Theorem}

This induces a (topological) group structure on $Gen$ which we will discuss in  more  detail later.

Let us make a couple of additional comments. The connection of  genericity with (non) forking is the following: a definable subset $X$ of $G$ is generic if and only if every translate $gX$ of $X$ does not fork over $\emptyset$.  We call this latter property $f$-genericity of $X$ (over $\emptyset$).

\subsection{Groups definable  in simple theories}
The understanding of simple theories and groups definable in simple theories developed in tandem with the understanding of specific examples, such as pseudofinite fields or ``smoothly approximated" $\omega$-categorical structures.   

As mentioned earlier there are very few stable theories, and  an impression given in the 1970's and early 80's, especially to those on the ``applied" side of model theory, was of stability as a rather obscure or quixotic subject in model theory, heavily set theoretic in character and with little meaning for the rest of mathematics. 

Shelah had studied what he called  simple theories already in the paper \cite{Shelah-simple}, and given two equivalent definitions: that any complete type does not divide over a small subset, and that no formulas has the ``tree property". 
Where $\phi(x,y)$ has the tree property (with respect to an ambient complete theory $T$) if there is $k$, and a tree $\{b_{\eta}: \eta\in \omega ^{<\omega}\}$ (finite sequences of natural numbers) such that for each $\eta \in \omega^{\omega}$, \{$\phi(x, b_{\eta|n}): n<\omega\}$ is consistent, but for each $\eta\in \omega^{<\omega}$, $\{\phi(x, b_{(\eta,{i})}:i\in\omega\}$ is $k$-inconsistent. 

In his Ph.D. thesis Byunghan Kim proved that  the ``algebraic"  properties of forking in stable theories hold for forking in simple theories.  And in subsequent work Kim and I (\cite{Kim-Pillay}) proved that in simple theories, stationarity (unique nonforking extension) of types over models is replaced by the ``Independence Theorem over a model": if $M$ is a model (elementary substructure of $\bar M$) and $A, B\supseteq M$ are $f$-independent over $M$, and $p_{1}(x)$, $p_{2}(x)$ are complete types over $A, B$ respectively, which do not fork over $M$ and have the same restriction to $M$, then $p_{1}, p_{2}$ have a common nonforking extension over $A\cup B$.   The expression ``Independence Theorem .." was already coined by and proved by Hrushovski (\cite{Hrushovski-pseudofinite})  for so-called $S_{1}$-theories, which  in hindsight are simple theories of ``finite $SU$-rank" with some additional definability properties. 

Let me describe the theory of groups definable in simple theories (the expression simple groups is too ambiguous), which was basically done in \cite{Pillay-definability-simple}, although I will change the terminology a bit. 
So we fix a group $G$ definable without parameters in a simple theory $T$. We will here reserve the expression ``generic" for definable subsets $X$ of $G$,  finitely many translates of which cover $G$.  

As introduced informally in Section 3.1 we define:
\begin{Definition} $X$ (a definable subset of $G$) is $f$-generic if every (left) translate $gX$ of $X$ by an element $g\in G$ does not fork over $\emptyset$
\end{Definition}

This will not in general correspond to {\em generic} in simple theories.

There are again several interrelated aspects; stabilizers, connected components, $f$-generics.  All have interesting differences with the stable situation.   However the chain conditions results from the stable case will drastically fail in the simple case, although I will not go into examples.

We fix $G$ definable in a saturated model of a simple theory, defined without parameters for ``simplicity". 
Fix a model $M$ of $T$ which may or may not be (relatively) saturated, and $p(x)\in S_{G}(M)$.  For $g\in G$, $gp$ is the partial type (over $M,g$) whose set of realizations is the left translate by $g$ of the set of realizations of $p$. As a partial type in variable $x$, $gp$ can  be described as $\{\phi(g^{-1}x): \phi(x) \in p\}$. 

We define $S(p)$ to be the set of $g\in G$ in the big model, such that the partial type $p(x)\cup gp(x)$ does not fork over $M$ (in the sense that no formula in the partial type forks over $M$). This is a type-definable over $M$ subset of $G$ (namely defined by a possibly infinite collection of formulas over $M$).  
The group-theoretic input of the independence theorem over a model is:
\begin{Proposition} $S(p)\cdot S(p)$ which by definition is the set of products $ab$ of elements $a,b\in S(p)$, is a {\em subgroup} of $G$ which is (clearly) infinitely-definable over $M$. 
\end{Proposition} 

This group $S(p)\cdot S(p)$ is traditionally called the stabilizer of $p$, $Stab(p)$, although it is not literally the set of $g\in G(M)$ such that $gp = p$. In fact as $Stab(p)$ is type-definable with parameters from $M$ it may have no points in $G(M)$ other than the identity. 

Now for the connection with $f$-generics. 
\begin{Proposition} (i) There exist $f$-generic types in $S_{G}(M)$ (that is types $p(x)\in S_{G}(M)$ such that every formula in $p$ is $f$-generic as defined above). 
\newline
(ii)  $p\in S_{G}(M)$ is $f$-generic if and only if $Stab(p)$ has ``bounded index" in $G$.
\end{Proposition}

Let us discuss some notions and issues in the proposition above.  $Stab(p)$ is an infinitely definable subgroup of $G = G({\bar M})$.  ``Bounded index" is an appropriate generalization of finite index (to type-definable subgroups).  One definition is that the quotient is the same when passing to larger (saturated) models.  Another is that the quotient has cardinality $\leq 2^{|M|+|T|}$.  In this context let us define $G^{00}_{M}$ to be the smallest type-definable over $M$ subgroup of $G$ of bounded index.  Then  in (ii) $Stab(p)$ will be $G^{00}_{M}$.

Hrushovski's stabilizer theorem for approximate subgroups takes place in a very different environment: a so called $\vee$-definable group in place of a definable group, and a certain (pseudofinite) measure on ``definable" subsets of $G$. But nevertheless, it is an analogue of Proposition 3.9 above rather than the fundamental theorem of stable group theory.  This is discussed in more detail in Section 4.3.

\vspace{5mm}
\noindent
\subsection{Groups definable in $NIP$ theories.}
The $NIP$ property for  first order theory $T$ was defined in Section 1.3.  As in the case of simple theories, $NIP$ theories include theories and structures which have been studied intensively using rather basic model-theoretic tools (quantifier elimination etc.). The idea that $NIP$ theories are also subject to stability-type machinery is again due to Shelah \cite{Shelah-dependent}. The papers \cite{NIPI}, \cite{NIPII} are also behind much of the material in this section.  

One of the main points on the purely model-theoretic side is that nonforking is connected with invariance:
\begin{Fact} Suppose $T$ is $NIP$. Let $p(x)$ be a complete type over the monster model ${\bar M}$, and let $M$ be a small elementary submodel of ${\bar M}$. Then $p$ does not fork over $M$ iff $p$ does not divide over $M$ iff $p$ is $Aut({\bar M}/M)$-invariant. 
\end{Fact}

Invariance is also called nonsplitting: $Aut({\bar M}/M)$ invariance of $p(x)\in S({\bar M})$ means {\em precisely} that for every $L$-formula $\phi(x,y)$ and tuple $b$ from ${\bar M}$, whether or not $\phi(x,b)$ is in $p(x)$ depends on $tp(b/M)$. So this is a weakening of {\em definability} of $p$ over $M$. 

The difference with stable/simple theories is that in $NIP$ theories it is NOT the case that any type over ${\bar M}$ does not divide over a small model $M$.  Nevertheless the whole machinery of invariant (over a small model) types is crucial for the general theory, as  well as applications. For example, the ubiquity of invariant types, and among them, definable types,  is behind the understanding of imaginaries in theories of valued fields \cite{HHM}.

In the $NIP$ environment we can improve slightly the invariance property. Given an $L$-formula $\phi(x,y)$ and model $M$, we have the topological space $S_{y}(M)$.  So it makes sense to talk about a type $p(x)\in S({\bar M})$ being {\em Borel definable} over $M$:  $p(x)$ is $M$-invariant and moreover for each $\phi(x,y)\in L$, the set of $q(y)\in S_{y}(M)$ such that $\phi(x,b)\in p$ for some (any) realization $b$ of $q$ is a Borel subset of the space $S_{y}(M)$. 

\begin{Proposition} Assume that $T$ is (or has) $NIP$.  Then
\newline
(i)  If $p(x)\in S({\bar M})$ does not fork over $M$ then $p(x)$ is Borel definable over $M$.
\newline
(ii) Suppose moreover that $T$ is countable (i.e. $L$ is countable). Then in (i) we may take the ``Borel definitions" to be countable unions of closeds  (namely $F_{\sigma}$). 
\end{Proposition}

We now pass to groups definable in $NIP$ theories. We fix such a group $G$, which we assume to be defined over $\emptyset$ in the monster model ${\bar M}$ of the $NIP$ theory $T$. 

We have already defined $G_{M}^{00}$ in the previous section. The first observation is that this does not depend on $M$ when $T$ is $NIP$, so we just write it as $G^{00}$.  More details about these connected  components will appear in Section 4.

It is natural to ask what, if anything, is the $NIP$ analogue of  Theorem 3.6 and Proposition 3.9 above  (i.e. stabilizers and $f$-generics).

In \cite{NIPII} we introduced terminology ($f$-generics ...) consistent with what we have said above in the stable and simple cases.  This terminology was changed in \cite{Chernikov-Simon}.  I will make a few compromises which will preserve consistency with the earlier material in this article, and hopefully not cause confusion.

\begin{Definition} Let $\phi(x)$ be a formula (possibly with parameters), defining a subset $X$ of $G$.  Let $A$ be a small set of parameters.  We will say that $X$ (or $\phi$) is $f$-generic over $A$, if every (left) translate $gX$ of $X$ by an element $g\in G$ does not fork over $A$.
\end{Definition}

\begin{Remark} (i) So what we called $f$-generic in Section 3.2 is now called $f$-generic over $\emptyset$ for formulas or definable sets.
\newline
(ii) $X$ is $f$-generic over some set $A$ of parameters over which $X$ is defined iff $X$ is $f$-generic over every set $A$ of parameters over which $X$ is defined  (if and only if whenever $(g_{i}:i<\omega)$ is indiscernible over the canonical parameter of $X$, then $\{g_{i}X:i\in\omega\}$ is inconsistent). 
\end{Remark}

As in \cite{NIPII} and \cite{Chernikov-Simon} we will typically work with global complete types of $G$, i.e. $p(x)\in S_{G}({\bar M})$. 

\begin{Definition} (i) $p(x)\in S_{G}({\bar M})$ is said to be weakly $f$-generic if every formula in $p$ is $f$-generic over some set of parameters.
\newline
(ii) We say that $p(x)\in S_{G}({\bar M})$ is strongly $f$-generic if there is some  small model $M$ (or even small set of parameters) such that every formula in $p$ is $f$-generic over $M$, namely  for all $g\in G({\bar M})$, $gp$ does not fork over $M$.
\end{Definition}

For simple theories, we will, just take it as a fact that what we call here a strongly $f$-generic type is the same as what we called an $f$-generic type (namely we can take $M = \emptyset$). Only the notion of strongly $f$-generic appeared in the paper \cite{NIPII} which was the orgin of this study of $NIP$ groups. Weak $f$-generics appear later in \cite{Chernikov-Simon}.

The existence of strongly (or weakly) $f$-generic types for groups $G$ in $NIP$ theories is bound up with a new notion in this article, {\em definable amenability}.  More on this will be said in the section below on topological dynamics. But for now, we will call a definable group $G$ (in an ambient theory) {\em definably amenable} if there is a ``left invariant Keisler measure on $G$". This means that there is a $[0,1]$ valued function $\mu$  on the Boolean algebra of definable subsets of $G$, which takes $G$ to $1$, $\emptyset$ to $0$ and such that if $X,Y$ are disjoint definable subsets of $G$ then $\mu(X\cup Y) = \mu(X) + \mu(Y)$, and such that $\mu(X) = \mu(gX)$ for all $g$ in $G$.  We are here working with definable (with parameters) subsets of $G$ in the monster model, but it is enough to consider definable subsets in a given model. 

We return in this section to thinking of or defining stabilizers of types literally, namely for a global type $p(x)\in S_{G}(\bar M)$, $Stab(p) = \{g\in G({\bar M}): gp = p\}$

We assume $NIP$ for the rest of this subsection. 
The following Fact is Proposition 5.6(i) of \cite{NIPII}.

\begin{Fact} If $p(x)\in S_{G}(\bar M)$ is strongly $f$-generic then  $Stab(p) = G^{00}$. 
\end{Fact}

\begin{Proposition} The following are equivalent:
\newline
(i) There is some strongly $f$-generic type $p(x)\in S_{G}({\bar M})$,
\newline
(ii) $G$ is definably amenable. 
\end{Proposition}

(ii) implies (i) is Proposition 5.9 of cite{NIPII} which is unproblematic.  However (i) implies (ii) appears as Proposition 5.6 (ii) of \cite{NIPII},  whose proof seems to have a gap,  as pointed out to me by Krzysztof Krupi\'nski.  So I will point out another route to (i) implies (ii), coming out of discussions with Krupi\'nski. I will start by giving some background.

First, obtaining definable amenability from suitable hypothesis uses Haar measure on the compact group $G/G^{00}$ in an interesting manner.  The basic assumption is the existence of a type $p(x)\in S_{G}(\bar M)$ such that $p$ does not fork over some small model $M$ and $Stab(p) = G^{00}$.  The quotient $G/G^{00}$ is a compact group, and so equipped with its Haar measure (a left and right invariant Borel probability measure) $\eta$.   We want to define a left invariant Keisler measure $\mu$ on $G$, using $\eta$. So pick a definable (with parameters) subset $X$ of $G$. Now as $Stab(p) = G^{00}$, whether or not $X\in gp$ depends only on the coset $g/G^{00}$, so we would like to define $\mu(X)$ to be the Haar measure of $\{g/G^{00}: X\in gp\}$, which if possible, yields the required invariant Keisler measure on $G$. But we need the last set to be Borel in $G/G^{00}$ (so measurable). Let $Y = \{g\in G: X \in gp\}$.  By Proposition 3.11 (i), $Y$ is Borel over $M$ in the obvious sense. For $\pi$ the canonical homomorphism $G\to G/G^{00}$, the set we are interested in is precisely $\pi(Y)$. (Alternatively we can replace $Y$ by the collection of $tp(g/M)$ for $g\in Y$, and replace $\pi$ by the well-defined continuous surjective map from $S_{G}(M)$ to $G/G^{00}$.) There is no reason that $\pi(Y)$ should be Borel.  However when $T$ is countable,  Proposition 3.11 (ii) gives the additional information that $Y$ is a countable union of closed sets, and this DOES yield that our set $\pi(Y)\subseteq G/G^{00}$ is also $F_{\sigma}$, so Borel. 
Hence the whole argument of deducing definable amenability involves passing to a countable sublanguage. 

In the proof of Proposition 5.6 (ii) of \cite{NIPII} that the existence of a strongly $f$-generic global type $p$  implies definable amenability, we said that  we CAN reduce to  $T$ being countable, which suggests that we thought that $p|L_{0}$ is also strongly $f$-generic, for any (suitably large) countable sublanguage $L_{0}$ of $L$. In fact this IS problematic, although maybe true. 

After this preamble, let us give a proof of (i) implies (ii). Let $p$ be a global strongly $f$-generic type.  By Fact 3.15, $Stab(p) = G^{00}$, in particular the orbit $Gp$ is bounded.  To prove that $G$ is definable amenable it is enough to prove it in all reducts of $T$ to countable sublanguages. So let $L_{0}$ be a countable sublanguage of $L$ (in which $G$ is definable).  Then $p|L_{0}$ also has bounded orbit under $G$, so by Theorem 3.12 of \cite{Chernikov-Simon} (or rather the proof of (3) implies (1) there)  there is a strongly $f$-generic global type in the reduct of $T$ to $L_{0}$.  The discussion above shows that $G$ is definably amenable in $L_{0}$.  This completes the proof of 3.16.

\vspace{5mm}
\noindent
It is an open question whether definable amenability of $G$ is equivalent to the existence of a global weakly $f$-generic type. On the other hand, under the definably amenable assumption, weakly $f$-generic types correspond to Newelski's ``weak generic" types \cite{Chernikov-Simon}.

There are many more things that could be said here, but we will defer to later sections.  But for now let us just mention that definably amenable groups in $NIP$ theories include solvable groups (as they are amenable as abstract, or discrete groups) but also so-called $fsg$ groups (see \cite{NIPII}) which are abstract versions of ``definably compact" groups, familiar from $o$-minimal theories.  Examples of non definably amenable $NIP$ groups are $SL(2,R)$, $R$ a big real closed field. It is also worth remarking that all stable groups are definably amenable. This can be seen in various ways. For example, by taking the average of the generic types in a suitable sense. Or because stable groups are $fsg$.  So the free group on more than $1$ generator is definably amenable (as a ``pure" group), but is the standard example of a non amenable group (qua discrete group).  See Section  7 for a discussion of the model theory of the free group.

\subsection{Geometric stability} 
The notions of geometric stability and geometric model theory cover lots of ground. At some point they referred to ``combinatorial geometries", namely matroids, but have since come to refer to methods informed by ``algebraic-geometric" rather than combinatorial sensibility. A lot of this article is about geometric model theory, in this more general sense. In this section I will focus on fairly special notions, $1$-basedness, or modularity, some of which have already appeared. I will concentrate on the stable case, although there are generalizations to simple, $o$-minimal, and other (including pseudofinite) environments.  Although the origin of this material is work and conjectures of Zilber, I will not really give an  historical account.  The reader can look at \cite{Pillay-GST}  for more details and references to original papers.  

\begin{Definition} Suppose that $p(x)\in S(M)$ is a complete type over a model $M$ (elementary substructure of monster ${\bar M}$) which is {\em definable}. Then the canonical base of $p$, $Cb(p)$ is the set of ``codes" in ${\bar M}^{eq}$ of the $\phi(x,y)$-definitions of $p$ as $\phi(x,y)$ ranges over $L$-formulas. 
\end{Definition} 

{\em Explanation.}  Given $\phi(x,y)\in L$ the $\phi$-definition of $p$ is a formula $\psi_{\phi}(y,c)$  (where I exhibit the parameters) such that for any $b\in M$, $\phi(x,b)\in p(x)$ iff $M\models \psi_{\phi}(b,c)$. Now any such formula $\psi_{\phi}(y,c)$ can be considered as an element in $M^{eq}$ (i.e. as $c/E$ for suitable $E$). And $Cb(p)$ is the collection of all such codes (as a set, sequence, or definable closure of such a set/sequence).

If $p'(x)\in S({\bar M})$ is the global extension of $p$ given by the same defining schema, then $Cb(p) = Cb(p')$ and is the unique set in ${\bar M}^{eq}$ which is fixed pointwise by any automorphism of ${\bar M}$ fixing $p'$.

As in a stable theory types over models are definable, all this applies to all types in stable theories. 

There is also a good notion of canonical base for arbitrary global types in simple theories, although such types need not be definable. But this depends (in general) on introducing hyperimaginaries, as well as the independence theorem for KP-strong types (some of which will be discussed later). 

Maybe at this point it is worth mentioning why the study of complete types is important, compared to just definable sets, especially for the notion of ``canonical base".  In most ``geometric" contexts,  we distinguish between closed and open sets, especially with regard to the Zariski topology ($ACF$) or Kolchin topology  ($DCF_{0}$). In these cases, we have the $DCC$ on closed sets, definable sets are Boolean combinations of closed/open sets, any closed set is a finite union of its irreducible components, and an (absolutely) irreducible closed set $X$ corresponds to a unique stationary type $p$,  the one axiomatized by $x\in X$ and $x\notin Y$ for $Y$ any proper closed subset of $X$.  In this case the canonical base of $p$ is the same thing as the canonical parameter  (or generator of smallest field of definition) for $X$.  But in general a stable theory does not support any abstract Zariski-topology, so the canonicity aspect of closed irreducible sets is replaced by stationary types. 

From now on we work with a stable theory $T$ and in  $T^{eq}$ (${\bar M}^{eq}$). 

We now give the right definition of $1$-basedness (of a partial type). 
\begin{Definition} (i) Let $\Sigma(x)$ be a partial type, over a set $A$ of parameters.  We call $\Sigma$ $1$-based, if whenever ${\bar a}$ is a tuple of elements realizing $\Sigma$, and $B$ is any algebraically closed set containing $A$ (which we could take to be a model), then  $Cb(tp({\bar a}/B))$ is contained in $acl({\bar a},A)$. 
\newline
(ii) $T$ itself is called $1$-based if each sort of $T$ (consider as a formula over $\emptyset$) is $1$-based (so $A$ disappears).
\end{Definition} 

When $T$ is a $1$-sorted structure, then for $T^{eq}$ to be $1$-based, it suffices for the ``home sort" of $T$ to be $1$-based. 

The expression $1$-based was introduced by Buechler because it says that the canonical base of a stationary type is in the algebraic closure of just one member of a Morley sequence in that type (rather than being in the algebraic closure of the whole Morley sequence).   Anyway $1$-basedness of a formula $\phi(x)$ (over $\emptyset$) says something about  ALL stationary types $tp({\bar a}/A)$ where ${\bar a}$ is a tuple whose coordinates satisfy $\phi(x)$. So pay attention. 

There is a  connection with the structure of strongly minimal sets in $T$. Remember that a definable set $X$ is strongly minimal if it is infinite and all definable subsets of $X$ are finite or cofinite. Equivalently $X$ has Morley rank $1$ and Morley degree $1$.  Strongly minimal sets are the ``building blocks" in a technical sense, of totally transcendental theories in which all types (or definable sets) have finite Morley rank, such as when $T$ is $|T|^{+}$-categorical.  In a strongly minimal set algebraic closure:  $a\in acl({\bar b}, A)$ for $a\in X$ and ${\bar b}$ a tuple from $X$, gives rise to a pregeometry or matroid, as well as a resulting notion of dimension and independence for tuples in a strongly minimal set $X$.  See  Chapter 2 of \cite{Pillay-GST} for more details, and in fact dimension and independence coincide with Morley rank and $f$-independence.

The strongly minimal set $X$ is said to be {\em modular} if for any finite tuples ${\bar a}$, ${\bar b}$ of elements of $X$, ${\bar a}$ is independent from ${\bar b}$ over $acl(\bar a)\cap acl({\bar b})\cap X$  (working over $A$). And $X$ is said to be locally modular if $X$ is modular working over $A$ together with some number (one is enough) of elements of $X$.

\begin{Lemma}  $X$ (strongly minimal) is locally modular iff $X$ is $1$-based.
\end{Lemma}

In fact $1$-basedness of $X$ is rather similar to modularity of $X$, but working in $X^{eq}$: $a$ and $b$ are independent over $acl^{eq}(a)\cap acl^{eq}(b)$  (over $A$). 

This notion of $1$-basedness and modularity involves another kind of complexity, which we sometimes call geometric complexity (initiated or discovered by Zilber).  For example given a strongly minimal set $X$ we can consider the collection of strongly minimal subsets of $X\times X$, in particular (uniformly) definable families of such. Local modularity of $X$ is equivalent to saying that for any definable family ${\mathcal F}$ of strongly minimal subsets of $X\times X$, for a ``general"  point 
$(a_{1},a_{2})\in X\times X$ only finitely many members of the family ${\mathcal F}$ pass through $(a_{1},a_{2})$.

When $X$ is an algebraically closed field, then the family $\{y = ax+b: a,b\in X\}$ witnesses non local modularity of $X$. 

The main point here, which is related to Section 2.4, is the properties of $1$-based groups (as definable groups in an ambient stable theory):
\begin{Proposition} Suppose $G$ is a $1$-based group (say $\emptyset$-definable in saturated model ${\bar M}$ of a stable theory). Then
\newline
(i) $G$ is commutative-by-finite (has a commutative subgroup of finite index),
\newline
(ii) Every type-definable connected subgroup of $G$ is type-defined over $acl^{eq}(\emptyset)$,
\newline
(iii) Every definable subset of $G$ is a Boolean combination of cosets (or translates) of definable subgroups of $G$,
\newline
(iv) Moreover this all holds for each group $G\times ...\times G$. 
\end{Proposition}

In the case when $T$ is $t.t$, in (ii) type-definable groups are definable, so it just says that connected definable subgroups are $acl(\emptyset)$-definable (a property we call rigidity).

The connection with the general theory of stable groups as in Section 3.1 is, (assuming that $T$ is $t.t$ for simplicity):
\begin{Remark} Suppose $G$ is $1$-based. Let $p(x)\in S_{G}({\bar M})$. Then $Stab(p)$ is a connected definable subgroup of $G$ definable over $acl(\emptyset)$, and $p$ is of the form $gq$ where $q\in S_{G}({\bar M})$ is the (unique) generic type of $Stab(p)$, and $g\in G({\bar M})$.

\end{Remark}

\section{$G/G^{00}$}

I   have been discussing above the connected components of a definable group $G$.   I want to make this more precise. There will be many connections with the following three sections.  These connected components seem rather innuocuous, but it turns out that a lot of information is contained in the quotients, especially in the case of pseudofinite groups.  

Again we fix a complete theory $T$, a saturated model $\bar M$ and a group $G$ $\emptyset$-definable in ${\bar M}$. As above, definability means with respect to $T$ (i.e. ${\bar M}$). 

\begin{Definition} Fix a small set $A$ of parameters.
\newline
(i) By $G_{A}^{0}$ we mean the intersection of all $A$-definable subgroups of $G$ of finite index.
\newline
(ii) By $G_{A}^{00}$ we mean the intersection of all type-definable over $A$, subgroups of $G$, with index at most 
$2^{(|A|+ |T|)}$  (equivalently the smallest such group).
\newline
(iii) By $G_{A}^{000}$ we mean the intersection of all $Aut({\bar M}/A)$-invariant subgroups of $G$ of index at most $2^{(|A| + |T|)}$ (equivalently the smallest such subgroup).
\end{Definition} 

\begin{Remark} (i)   In place of specifying the explicit bound on the index in (ii) and (iii) one could just say ``bounded index" meaning that changing the monster model will not change the quotient. And the reader should check the ``equivalently" statements.
\newline
(ii) For suitable theories (including stable theories) these connected components do not depend on the choice $A$ of the parameter set, in which case we just write $G^{0}$ etc.
\newline
(iii) Part (i) of the definition is transparent: In many situations such as groups definable in $o$-minimal structures or in totally transcendental structures, $G$ has a smallest definable subgroup of finite index, which deserves the name of the connected component of $G$. And it actually is such, in the topological sense, when $G$ is a Lie group definable in $(\R,+,\times)$ or  an algebraic group defined in $(\C,+,\times)$.  The generalization as in (i) of the definition, is also natural, as witnessed by the case where $G$ is (a nonstandard version of) $(\Z,+)$. 
\end{Remark}

It is natural to ask why and how connected components in the sense of (ii) in Definition 4.2 arise.  In the context of groups definable in simple theories, the role of $G^{0}$ for stable groups had to be replaced by the more general $G_{A}^{00}$, on general grounds,  although we still do not know an example where $G^{00}_{A}$ does not equal $G^{0}_{A}$.  These connected components also  naturally arise in a ``nonstandard analysis" setting in the form of the standard part mapping.  Recall that if $X$ is a compact (Hausdorff) space, considered for example as a set living in the standard model $\mathbb V$ of set theory, then we can pass to a saturated elementary extension ${\mathbb V}^{*}$ of $\mathbb V$, where there is a nonstandard version $X^{*}$ of $X$, and we have the standard part map $st_{X}:X^{*}\to X$.  When $X$ is a compact group, then $X^{*}$ is a definable group in ${\mathbb V}^{*}$ and $st_{X}$ will be a surjective group homomorphism, 
whose kernel $ker(st_{X})$ is the group of elements of $X^{*}$  infinitesimally close to the identity and in general (properly) contains 
$(X^{*})_{\mathbb V}^{00}$.  As long as the compact group $X$ is not profinite $(X^{*})_{\mathbb V}^{00}$ will be different from
$(X^{*})_{\mathbb V}^{0}$.  Of course the theory of $({\mathbb V}, \epsilon)$ is not particularly well-behaved model theoretically, but this kind of phenomenon appears widely in unstable $NIP$ theories. 

The connected component in (iii) of Definition 4.1 ($G$ with three zeros as certain Polish model theorists say) was first defined (by me I believe) abstractly, in analogy with the full Lascar group, and it took some time for an example to be found where it was different from $G^{00}$, at least in a tame environment (\cite{Conversano-Pillay}). 

Note that from the definitions we have that $G^{0}_{A}\supseteq G^{00}_{A} \supseteq G^{000}_{A}$. 

Let us now discuss the {\em logic} topology on these quotients which is implicit in earlier sections. The basic definition is:
\begin{Definition} Let $X$ be a definable set (or sort). Let $E$ be a type-definable  equivalence relation on $X$,  all (type)-defined over $A$.  We again call $E$ bounded if there are boundedly many (or $\leq 2^{(|T| + |A|)}$ many classes). Assume $E$ is bounded and let $\pi:X\to X/E$ be the canonical projection. Call a subset $Z$ of $X/E$ closed if $\pi^{-1}(Z)$ is type-definable (over some small set of parameters).
\end{Definition}

The following is an exercise using the compactness theorem, but the reader can look at \cite{Pillay-typedefinability}

\begin{Lemma} (i) In the context of Definition 4.1, the logic topology is compact and Hausdorff.
\newline
(ii) If moreover $X$ is a definable group $G$, and $E$ is an equivalence relation induced by a normal type-definable subgroup, then $X/E = G/H$ is a compact (Hausdorff) topological group, under the logic topology.
\end{Lemma}

In general we call quotients of definable sets by type-definable equivalence relations {\em hyperdefinable} sets, and their elements {\em hyperimaginaries}. So the point is that a bounded hyperdefinable set has the canonical structure of a compact topological space (generalizing the fact a finite set is a compact topological space under the discrete topology).
Let me just remark that another approach to this logic topology is to observe that (with earlier notation)  if $M$ is an elementary substructure of ${\bar M}$ containing $A$, then the map $\pi: X \to X/E$ factors through the tautological map  $a\to tp(a/M)$ from $X$ to $S_{X}(M)$, and that the logic topology on $X/E$ coincides with the quotient topolology on the type space $S_{X}(M)$. Some things have to be checked.

\begin{Lemma}  $G_{A}^{0}$,  $G_{A}^{00}$, and $G_{A}^{000}$ are normal subgroups of $G$. The canonical homomorphism $G/G_{A}^{00}  \to G/G_{A}^{0}$ is continuous, and in fact $G/G_{A}^{0}$ is the maximal profinite quotient of $G/G_{A}^{00}$. 
\end{Lemma}

I don't want to get into continuous logic here, but it is implicit in the above notions. For $E$ a bounded type-definable equivalence relation on $X$ one can think of $\pi:X\to X/E$ as a formula with values in the compact space $X/E$ (rather than the space $\{0,1\}$). 

The relation of this whole topic to Lie groups is that a compact group is an inverse limit of compact Lie groups. In fact if $G$ is a compact group then $G^{0}$ (its topological connected component, which should not be confused with $G^{0}$ for $G$  a definable group) is an inverse limit of connected compact Lie groups, and of course $G/G^{0}$ is profinite.  Moreover a compact group is Lie iff it has the $DCC$ on closed subgroups. 

This section will be completed with a few more specialized subsections.
\subsection{The so-called ``Pillay conjecture"}
This conjecture or question was connected with trying to see how hyperimaginaries fit into contexts other than simple theories, as well as how the standard part map described above can be put into a more general model-theoretic environment.
It was also related to the idea of measuring the richness of the category $Def(T)$ for a theory $T$, by what compact spaces/groups can be obtained from $Def(T)$ by the constructions above.

The conjecture was that if one has a definable group $G$ in an $o$-minimal theory which is ``definably compact", and say definably connected $G = G^{0}$, then $G/G^{00}$ is a connected compact Lie group which has the same dimension as a Lie group as the $o$-minimal dimension of $G$. So the canonical map $\pi:G\to G/G^{00}$ can be viewed as an intrinsic standard part map, and $G$ can be viewed as a nonstandard version of the compact Lie group $G/G^{00}$.

This was proved in  \cite{NIPI}, depending on a series of papers and results by the authors and others, including Edmundo, Berarducci, and Otero.  It was partly through this work that Keisler measures were rediscovered, and it also impacted on the general theory of groups definable in $NIP$ theories. In fact the notion of an $fsg$ group (group with finitely satisfiable generics) came  directly out of the work on the Pillay conjecture, and in the $o$-minimal context these are precisely the ``definably compact" groups. 

In the $o$-minimal context, definably compactness of a definable group is equivalent to saying that the group, as a definable set (subset of $M^{n}$) for some model $M$ such that multiplication is continuous (with respect to the topology induced by the $o$-minimal structure), is both closed and bounded.

\subsection{Compact domination}
A variant of the ``Pillay conjecture" is the ``compact domination conjecture" which was articulated in discussions with Hrushovski in 2005 (during the Newton Institute semester on model theory).  The language is motivated by the notion of stable domination which appears in the literature on algebraically closed valued fields (\cite{HHM-book}).   It is also an  attempt to describe model-theoretic features of the standard part map, at least in certain ``tame" environments. 
The basic environment is that of a $\emptyset$-definable group $G$ (in an ambient theory $T$ or saturated model $\bar M$) such that $G^{00}$ is well-defined (namely independent of the choice of a base set of parameters). And $\pi$ is the usual map from $G$ to $G/G^{00}$. 
\begin{Definition} $G$ satisfies compact domination (namely is dominated by $G/G^{00}$) if for any definable (with parameters) subset $X$ of $G$, for almost all $c\in G/G^{00}$ (in the sense of Haar measure), $\pi^{-1}(c)$ is either contained in, or disjoint from $X$.
\end{Definition} 

Compact domination was proved for definably compact groups in $o$-minimal structures in a series of papers (\cite{NIPI}, \cite{NIPII} and \cite{Centralextensions}).  It was generalized to suitable  groups in so-called distal structures by Pierre Simon, but I do not know a precise reference.

There is a certain weak version of compact domination called ``generic compact domination"  which says that for any definable subset $X$ of $G$, for almost all $c\in G/G^{00}$, not both $X\cap \pi^{-1}(c)$ and its complement in $\pi^{-1}(c)$ are ``large" in $\pi^{-1}(c)$. 

For appropriate notions of large this has been proved for suitable groups in $NIP$ theories. See later for references and the  connection with arithmetic regularity theorems.

\subsection{Approximate subgroups}
I will not go into detail about  all the results (by Hrushovski, Breuillard, Green, Tao) but just mention the connection with themes of this section.  We will take the definition of a $k$-approximate group to be a finite symmetric ($A = A^{-1}$) subset $A$ of a group $G$, such that  the cardinality of the set $AA$ of products $ab$ for $a,b\in A$ is at most $k$ times the cardinality of $A$.  The problem is to describe the structure of $k$-approximate groups. 

We follow the ``nonstandard analysis" strategy of  Hrushovski \cite{Hrushovski-approximate}. 

We can consider all pairs $(G,A)$ where $A$ is a (finite) $k$-approximate subgroup of a (not necessarily finite) group $G$, and then consider a model-theoretic limit (e.g. ultraproduct) of all these. The ambient structure where this lives could even be a nonstandard model ${\mathbb V}^{*}$ of set theory. 
Let us simply call this new pair also $(G,A)$.  We assume that the ambient model is saturated and $(G,A)$ is defined over a small submodel $M$.  Let $H$ be the subgroup of $G$ generated by $A$. $H$ is what is called a $\vee$ or $ind$-definable group (or subgroup of $G$).  $H$ need not be itself definable, but it is an increasing union of definable pieces  $A....A$, and by a (type-) definable subset of $H$ we mean a (type-) definable subset of one of those definable pieces. 

Now in contradistinction to the case of definable groups, there is no reason that a $\vee$-definable group such as $H$ should have any nontrivial definable or type-definable subgroups.   However, the assumptions at hand, in particular the existence of an invariant measure $\mu$ on definable subsets of $H$ (the pseudofinite counting measure) as well as the $k$-approximateness of $A$ (in terms of $\mu$ in the obvious sense), help in proving:

\begin{Proposition} $H_{M}^{00}$ exists.
\end{Proposition}

This means that there exists a type-definable over $M$ subgroup of $H$ of bounded index, whereby $H_{M}^{00}$ is the smallest one.  This is called the stabilizer theorem (as named by Hrushovski, and maybe more appropriate  than the expression ``Lie model theorem"), as it is obtained as $Stab(q)$ for a suitable type $q$, where $Stab$ corresponds more to the simple case than the stable or $NIP$ cases, and the proof is also related to the simple case. (See  Section 3.2 above). 

Now we can define, in a similar fashion as before, a ``logic topology" on $H/H^{00}_{M}$ which will equip it with the structure of a {\em locally compact group}.  A locally compact group is heuristically an inverse limit of Lie groups (more precisely has an open subgroup $O$ such that arbitrarily small neighbourhoods of the identity contain closed normal subgroups $N$ of $O$ such that $O/N$ is a Lie group). And this yields the basis for a description of $A$ (which gives a description of the $A$'s from the original family).  An important part of the final results of Breuillard, Green, Tao, \cite{BGT}, is that the Lie groups appearing above (for the locally compact group $H/H_{M}^{00}$) are nilpotent-by-finite.

\subsection{Bohr compactification}
This  is closely related to the topological dynamics section below, but we are introducing the notions first  in this subsection  because they fit thematically.  
 
Compactifications, of one kind or another, or in various categories, play a big role in mathematics partly because of the importance of compact spaces (and groups).

Anyway, we are going to work with discrete (or abstract) groups $G$. By a (group) compactification of $G$, we mean a compact (Hausdorff) group $C$ together with a homomorphism $f$ from $G$ to $C$ such that the image $f(G)$ is dense in $C$.  There exists a {\em universal} compactification of $G$, which we will call $b(G)$ (the Bohr compactification of $G$), namely  any  compactification $f:G\to C$ of $G$ factors uniquely through the map $G\to b(G)$.

The only point I want to make is that this can be described model-theoretically, but where the ambient theory is set theory. 
Namely consider $G$ as a living in  the standard model $\mathbb V$ of set theory. Let ${\mathbb V}^{*}$ be a very saturated elementary extension. Then $b(G)$ is precisely $G^{*}/(G^{*})_{\mathbb V}^{00}$ with the homomorphism from $G$ induced by the inclusion of $G$ in $G^{*}$ and the canonical homomorphism from $G^{*}$ to  $G^{*}/(G^{*})_{\mathbb V}^{00}$. 

If one is not comfortable with saturated models of set theory (as it is not clear what this means), we could just work instead with the structure $M$ consisting of $G$ with predicates for all subsets of $G, G\times G, ...$.

Hence for a group $G$ in a saturated model of an arbitrary theory $T$, and a small model $M$ over which $G$ is defined, we can and do think of the compact group $G/G_{M}^{00}$ as the ``definable Bohr compactification" of the group $G(M)$.  In fact this is the universal compactification in a suitablly defined category of definable compactifications of $G(M)$.

It is natural to ask if the quotient $G/G_{M}^{000}$ is also something with meaning for discrete groups. On the face of it the definition is very model-theoretic, involving $G({\bar M})$ and automorphisms of ${\bar M}$ over $M$. But in fact, as in \cite{Krupinski-Pillay}, there is another description. Again we are working in the model-theoretic environment of $T$, ${\bar M}$, $G$,  $\emptyset$-definable, and small model $M$. We first want to give a definition of $p\cdot q$ for $p, q\in S_{G}(M)$ which does  not involve computing things in the saturated model:  so $p\cdot q$ is the intersection of all $X\cdot Y$, for $X,Y$ definable sets in $p, q$ respectively, as a closed set in $S_{G}(M)$. 

\begin{Fact} Consider surjective maps $f$ from $S_{G}(M)$ to groups $H$ say, such that for $p,q\in S_{G}(M)$, $f$ has a common value on $p\cdot q$ which is precisely the product $f(p)\cdot f(q)$ computed in $H$.
Then $G/G^{000}_{M}$ (together with the map from $tp(g/M)$ to $g/G_{M}^{000}$) is the universal such map $f$.
\end{Fact}

In the case of a discrete group $G(M)$ living in the standard model $M = ({\mathbb V},\epsilon)$ of set theory (and where $T = Th(M)$), $S_{G}(M)$ is precisely $\beta (G(M))$ (the Stone-Cech compactification of $G(M)$) and Fact 4.8 applies, giving a ``new" invariant of the discrete group $G(M)$.  It would be reasonable to call this group the  ``Bohr quasi-compactification of $G(M)$".   In Section 6 we will tallk about a certain ``generalized Bohr compactification" of $G(M)$, but this will be compact (Hausdorff). 

\subsection{Pseudofinite groups and arithmetic regularity}
A pseudofinite group, is, loosely speaking, an infinite model of the theory of finite groups, or for those who like ultraproducts, something elementarily equivalent to an ultraproduct of finite groups (which happens to be infinite).  There are various ways of treating a pseudofinite group from the model-theoretic point of view;  as just a group (in the group language),  or as a group definable in a nonstandard model ${\mathbb V}^{*}$ of set theory (which believes the group to be finite),  or simply a discrete group (so living in the standard model $\mathbb V$ of set theory. 

One of the main results (relevant to this section) is:
\begin{Proposition} Let $G$ be a pseudofinite group, considered as a group definable in a structure $M$, in any of the above senses. Then the definable Bohr compactification of $G$ is commutative-by-profinite, in other words the connected component of the compact group $b(G)$ is commutative.
\end{Proposition}
So in particular vewing $G$ as a discrete group, its Bohr compactification is commutative-by-profinite. The references for Proposition 4.9 are \cite{Pillay-compactifications-pseudofinite} and \cite{Nikolov} although other proofs have been found in discussions with Hrushovski.

This result together with the earlier material around domination are ingredients in recent work around ``arithmetic regularity" theorems in a tame environent (\cite{CPT1}, \cite{CPT2}) which we discuss briefly. 

In analogy with the  graph regularity theorem of Szemeredi,  there is a group version which  concerns pairs $(G,A)$ where $G$ is a finite group and $A$ an arbitrary subset. Can one say anything nontrivial about the family of $(G,A)$ (as $G$ ranges over arbitrary finite group, and $A$ over arbitrary subsets)? The main paper in the general environment is by Ben Green (\cite{Green}) which only considers the case of commutative $G$. The result is Fourier analytic and difficult to parse (at least for me).   We still have no theorem (or even conjecture) describing the situation when $G$ is not necessarily commutative. 

However let us consider ``tame"  versions of arithmetic regularity theorems, where more assumptions are made on the subset $A$ of $G$. One is that the formula $xy\in A$ is $k$-stable for some fixed $k$. A weaker one is that the formula $xy\in A$ is $k$-$NIP$.   Passing to the nonstandard environent  (ultraproducts for example) we obtain $(G,A)$ definable in a structure $M$, where $xy\in A$ is $k$-stable or $k$-$NIP$, and $G$ is equipped with the pseudofinite counting measure. 
I will state the main result of \cite{CPT2} and then discuss terms and ingredients of the proof. 
\begin{Proposition}Fix $k$ on which everything else depends.   For every $\epsilon$ there is $N$ such that  for all pairs $(G,A)$ 
where $G$ is a finite group and $A$ is a $k$-$NIP$ subset (i.e.  the formula $xy\in A$ is $k$-$NIP$), there is a normal subgroup $H$ of $G$ of index at most $N$, such that there is a $(\delta, r)$ Bohr-neighbourhood $B$ of  $H$, where 
$r\leq N$, and $1/N \leq \delta \leq 1$, such that for some subset $Z$ of $G$ of size $< \epsilon|G|$, after throwing away $Z$, $A$ is a union of at most $N(2/\delta)^{r}$ translates of $B$, up to a set of cardinality $<\epsilon|B|$.
\end{Proposition} 

By a $(\delta, r)$-Bohr neighbourhood of $H$ above we mean the preimage of a homomorphism $g$ from $H$ to the $r$-dimesional torus $S_{1}^{r}$,  of a ball around the identity of radius $\delta$.   The intervention of tori  (connected commutative compact Lie groups) here is due to Proposition 4.9 and the fact that a connected commutative compact group is an inverse limit of connected commutative compact Lie groups (which are precisely finite products of the circle group). 

The proof involves in addition a local generic domination theorem, proved in \cite{CP}.  From these ingredients we obtain a single theorem about pseudofinite groups $G$ together with a subset $A$ such that the formula $xy\in A$ is $NIP$, and this is pulled down, by routine methods, to obtain the statenent about all the finite pairs $(G,A)$.

There is a much stronger result when $xy\in A$ is $k$-stable.  In Proposition 4.10, $B = H$, and $Z=\emptyset$, so $A$ is a union of translates of the bounded index subgroup $H$ up to a set of size $<\epsilon|H|$.  This result in \cite{CPT1} depends on the existing ``local stable group theory"  (\cite{Hrushovski-Pillay-groupsinlocalfields}). Local here means only assuming stability of a single formula.  The first theorems on stable group regularity were by Terry and Wolf \cite{TW} but assumed $G$ to be a finite-dimensional vector space over $\F_{p}$  ($p$ fixed) and used just finitary methods.

\section{Galois theory}
The expression ``Galois theory"  carries quite a bit of baggage. The study of automorphism groups of arbitrary structures should not in itself be considered as part of Galois theory (thus the expression ``Galois type" in Abstract Elementary Classes is rather unfortunate). The existence of some kind of structure on the automorphism group and a suitable Galois correspondence are characteristic features of an abstract Galois theory.

The usual Galois theory of fields (which I will assume to be of characteristic $0$ for simplicity) has two aspects. First, given a 
polynomial equation $P(X) =0$ over a field $K$, we consider the (finite) set $Z = \{a_{1},...,a_{n}\}$ of solutions in a 
given algebraic closure $K^{alg}$ of $K$, and we consider the group $G$ of permutations $\sigma$ of $Z$ such that for 
every polynomial $Q(X_{1},...,X_{n})$ with coefficients from  $K$,  $Q(a_{1},..,a_{n}) = 0$ iff $Q(\sigma(a_{1}),..,\sigma(a_{n})) = 0$. 
This is of  course a finite group which we write as $Gal(P/K)$. We could instead consider the field  $L$ generated by $Z$ over $K$, and look at the group of field automorphisms of $L$ which fix $K$ pointwise. We write this as $Gal(L/K)$  (the usual notation for Galois groups). Then it is easy to see that these two  groups, $Gal(P/k)$ and $Gal(L/K)$ coincide, in the sense that the restriction of any $\sigma\in Gal(L/K)$ to the set $Z$ is in $Gal(P/K)$, and that any $\sigma$ in $Gal(P/K)$ extends uniquely to some $\sigma'\in Gal(L/K)$. 

We have the well-known Galois correspondence between subgroups of $Gal(L/K)$ and fields in between $K$ and $L$: to a subgroup $H$ of $Gal(L/K)$ we associate the fixed field of $H$ etc.  

Poizat \cite{Poizat-imaginaries} put all of this into a general model theoretic context, of a definably closed set $A$ in some
structure $M$, and a formula $\phi(x)$ over $A$ with finitely many solutions in $M$.  We have $Gal(\phi/A)$, the group of permutations of the  set $Z$ of solutions of $\phi(x)$ which are elementary over $A$, i.e. preserve all formulas with parameters from $A$. And in place of $L$ we have $B = dcl^{eq}(A)$, and $Gal(B/A)$. And there is also a Galois correspondence between subgroups of $Gal(B/A)$ and definably closed sets in between $A$ and $B$. Poizat also points out the role of elimination of imaginaries in the fields case.

Let us now work with a theory $T^{eq}$ and take $A = \emptyset$, and consider the collection of  elements of a model $M = M^{eq}$ (or even of a saturated model ${\bar M}^{eq}$) which are in $acl(\emptyset)$, namely which are realizations of some formula over $\emptyset$ with only finitely many other solutions.  This is precisely $acl(\emptyset)$. Now consider the group of all permutations of $acl(\emptyset)$ which are elementary (preserve all formulas).  This is precisely the Shelah group $Gal_{Sh}(T)$ of $T$.  It is the inverse limit of all the $Gal(\phi/\emptyset)$ where $\phi$ is an $L$-formula with finitely many realizations, and as such is a profinite group. 
Two tuples $a,b$ are said to have the same Shelah strong type (over $\emptyset$),  if they have the same type over $acl^{eq}(\emptyset)$, or equivalently $E(a,b)$ for every finite (finitely many classes) $\emptyset$-definable equivalence relation.  So we can consider $Gal_{Sh}(T)$ as the group of elementary permutations of the collection of Shelah strong types.

The $KP$-group $Gal_{KP}(T)$ is the generalization of this from algebraic imaginaries to bounded hyperimaginaries. As discussed earlier by a hyperimaginary we mean $a/E$ where $a$ is a tuple of length at most $\omega$ and $E$ is a type-definable over $\emptyset$ equivalence relation on tuples of the same length as $a$, and  $a/E$ has boundedly many images under automorphisms of ${\bar M}$. (One could also ask that $E$ itself has boundedly many classes.) There is no formal analogue for adjoining hyperimagaries as elements of new sorts (while still remaining in classical $\{0,1\}$-valued first order logic).  Nevertheless we let $bdd^{heq}(\emptyset)$ denote the collection of all bounded hyperimaginaries. The group of permutations of $bdd^{heq}(\emptyset)$ induced by $Aut({\bar M})$ is what we call $Gal_{KP}(T)$.  Two tuples have the same $KP$ strong type if they have the same type over $bdd^{heq}(\emptyset)$ (which one can make sense of)  iff $E(a,b)$ for all bounded type-definable over $\emptyset$ equivalence relations. And $Gal_{KP}(T)$ can be also seen as the group of elementary permutations of the collection of $KP$-strong types.  $Gal_{KP}(T)$ has naturally the structure of a compact group, and $Gal_{Sh}(T)$ is its maximal profinite quotient. 

For more details on hyperimaginaries, and Galois groups, see the original papers \cite{HKP}, and \cite{LP}. The book \cite{Casanovas-book} is also a good resource, especially the theorems regarding ``elimination" of hyperimaginaries.

Finally we have the full Lascar group $Gal_{L}(T)$.  To define it in analogy with $Gal_{Sh}$ and $Gal_{KP}$, we can first define ``bounded ultraimaginaries", these are equivalence classes $a/E$ of tuples with respect to $Aut({\bar M})$-invariant equivalence relations which have boundedly many classes. And $Gal_{L}(T)$ is the group of permutations of the collection of bounded ultraimaginaries induced by $Aut({\bar M})$.  Again we have Lascar strong types: $a$, $b$ have the same Lascar strong type if $E(a,b)$ for every bounded $Aut({\bar M})$-invariant equivalence relation, and $Gal_{L}(T)$ is the group of permutations of all Lascar strong types, induced by $Aut({\bar M})$. 

The Shelah, $KP$, Lascar, equivalence relations on tuples of the same length are denoted $E_{Sh}$, $E_{KP}$ and $E_{L}$.
 So $E_{L}$ refines $E_{KP}$ refines $E_{Sh}$ and we have surjective continuous homomorphisms  $Gal_{L}(T) \to Gal_{KP}(T) \to Gal_{Sh}(T)$. 

These Galois groups are invariants of the first order theory $T$ (up to bi-interpretability). 

All of this is dependent on parameters, and after naming a model by constants, the equivalence relations reduce to having the same type, and the Galois groups are all trivial. 

What, if anything, is the interest of such equivalence relations and Galois groups?
One concerns amalgamation of types: suppose $p(x)$ is a complete type over $\emptyset$.  Let $a$ realize $p(x)$. Let us suppose that $p(x)$ does not fork over $\emptyset$, so in particular for any given model, $p(x)$ has an extension $q(x)\in S(M)$ such that $q$ does not fork over $\emptyset$. But we can say even more; namely that we can choose $q(x)$ to be $tp(b/M)$ where $a$ and $b$ have the same Lascar strong type over $\emptyset$. As remarked above $q(x)$ determines the Lascar strong type of $b$ over $M$ so also over $\emptyset$.  
Hence Lascar strong types (hence $Gal_{L}(T)$) are an obstruction to both uniqueness of nonforking extensions and amalgamation of nonforking extensions.  (So in spite of the ``infinitary logic" appearance, $Gal_{L}(T)$ does belong to first order finitary model theory.) On the other hand there are theorems saying that these are the only obstructions.
The first is called the finite equivalence relation theorem, due originally to Shelah, and the second is called the Independence theorem for $KP$-strong types (due to Kim and me). 
\begin{Fact} (i) Assume that $T$ is  stable.  Suppose $a$, $b$ are tuples each of which is $f$- independent from a given set $A$ over $\emptyset$, and suppose that $a$ and $b$ have the same Shelah strong type over $\emptyset$, then $tp(a/A) = tp(b/A)$.
\newline
(ii) Assume $T$ is simple. Suppose that $A$ and $B$ are $f$-independent over $\emptyset$, and that we have tuples $a, b$ such that $a$ is $f$-independent from $A$ over $\emptyset$ and $b$ is $f$-independent from $B$ and $a$ and $b$ have the same $KP$-strong type over $\emptyset$. Then there is $c$ such that $c$ is $f$-independent from $A\cup B$ over $\emptyset$, and $tp(c/A) = tp(a/A)$ and $tp(c/B) = tp(b/B)$. 
\end{Fact}

The Lascar group was first  introduced by Lascar in \cite{Lascar} as a possible obstruction to recovering an $\omega$-categorical theory from its category of models. His original definition of $Gal_{L}(T)$ was as the quotient of $Aut({\bar M})$ by the normal subgroup generated by the fixators of small elementary submodels. $T$ is called ``$G$ compact"  if $Gal_{L}(T) = Gal_{KP}(T)$. Equivalently if Lascar strong types coincide with $KP$-strong types. It was open for a long time whether there were non $G$-compact theories. An example appears in \cite{CLPZ}, and there are now many more examples.

For stable theories all the groups and equivalence relations coincide. Simple theories are $G$-compact, but it is not known whether $KP$-strong types coincide with Shelah strong types. 

A lot of work has been done around how to think of $Gal_{L}(T)$ as a mathematical object. One point of view, which was fairly successful is as an object of   descriptive set theory.  This was already suggested in \cite{CLPZ}. Assume $T$ to be countable, and let $M$ be a countable model, and let $m$ be an enumeration of $M$. Let $S_{m}(M)$ be the space of complete types over $M$ extending $tp(m/\emptyset)$.   Then for $\sigma\in Aut({\bar M})$ the image of $\sigma$ in $Gal_{L}(T)$ depends only on $tp(\sigma(m)/M)$, so there is a canonical surjective map from $S_{m}(M)$ to $Gal_{L}(T)$, which turns out to be obtained by quotienting the Polish space by a Borel, in fact $F_{\sigma}$ equivalence relation. The question is  of the complexity of this equivalence relation from the point of view of the theory of Borel equivalence relations. A conjecture ``smooth implies closed" was made and solved in 
\cite{Kaplan-Miller-Simon} (and later also in \cite{KPR}    using methods of topological dynamics).

\vspace{5mm}
\noindent
I will now discuss another model-theoretic aspect of Galois theory, and that is ``definable automorphism groups", which (as Poizat first observed) subsumes the Galois theory of linear differential equations.  This concrete example of differential Galois theory is not so well-known, and I will leave its discussion to later.

Definable automorphism groups have their origin in work of Zilber (for example \cite{Zilber}) on the fine structure of $\aleph_{1}$-categorical theories, as well as problems about totally categorical theories. But it was put into a general form by Hrushovski. The stable case is given in detail in Chapter 7 of \cite{Pillay-GST}. But the  theory is developed in even greater generality in  \cite{Hrushovski-Bedlewo}.

For simplicity I will stick with the case when the ambient theory $T$ is totally transcendental (i.e. every formula has ordinal valued Morley rank), and I will work with formulas rather than types. 
\begin{Definition} Let $A$ be a small set of parameters, and $P, Q$ be $A$-definable sets (in the ambient saturated model ${\bar M}$). 
$Q$ is said to be $P$-internal, or internal to $P$, if there is some small set $B$ of parameters such that $Q\subseteq dcl(P,B,A)$, i.e. every element of $Q$ is in the definable closure of $A\cup B$ together with some tuple of elements from $P$.
\end{Definition} 

The compactness theorem together with stability yields:
\begin{Lemma}   Suppose $P$, $Q$ are $A$-definable sets and $Q$ is $P$-internal. Then there is a finite tuple $\bar b$ of elements in $Q$ and some $A$-definable function $f(-,-)$ such that for each $a\in Q$ there is a tuple ${\bar c}$ of elements in $P$, such that  $a = f({\bar c}, {\bar b})$. 
\end{Lemma} 

\begin{Remark} (i) The tuple ${\bar b}$ is called a {\em fundamental system of solutions} (relative to the data $P, Q$).
\newline
(ii) It may very well happen that $P$ = $\emptyset$, in which case $Q$ has to be a finite set and ${\bar b}$ can be taken as an enumeration of $Q$. This is precisely the case in the description of Galois theory at the beginning of this Section, which we call the {\em absolute case}. 
\end{Remark}

In the absolute case we had $Gal(Q/A)$ a finite group. In the more general case we will have $Gal(Q/P,A)$ which will be a {\em definable} group.  There are a few subtleties here, so I will go into a few details, but without proper proofs. 
Basically this definable group $Gal(Q/P,A)$ will have at least two incarnations (which will be definably isomorphic possibly over additional parameters). 
 The {\em intrinsic} version of the Galois group will be an $A$-definable group together with an $A$-definable action on the set $Q$, such that $G$ and the action are isomorphic to the group $Aut(Q/P,A)$ of permutations of $Q$ induced by automorphisms which fix $P$ and $A$ pointwise (equivalently which preserve all formulas with parameters from $P$). This is the group defined in  Section 4 of Chapter 7 of \cite{Pillay-GST}, and the group lives in $Q^{eq}$. 

The {\em extrinsic} Galois group depends on the choice ${\bar b}$ of a fundamental system of solutions: for $\sigma\in 
Aut(Q/P,A)$, $\sigma({\bar b}) = {\bar f}({\bar c}_{1},..,{\bar c}_{k},{\bar b})$, for some ${\bar c}_{1},..,{\bar c}_{k}$ 
from $P$.  This will induce a group operation on  the set of such tuples from $P$ which appear, quotiented by a suitable definable equivalence relation. This group lives on $P^{eq}$ and is definable over  $dcl({\bar b})\cap P^{eq}$,  a finite tuple from $P$. It is isomorphic (by definition) to $Aut(Q/P,A)$, but the action on $Q $ needs the additional parameter ${\bar b}$ to be defined. 

The two  incarnations of $Gal(Q/P,A)$ are definably isomorphic over $(A,{\bar b})$, but we need only  $(A, dcl({\bar b})\cap P^{eq})$ when the group is commutative.

As mentioned earlier, the Galois group also can appear in a couple of ways, as a group of permutations of a definable set, or as a group of permutations of a set of parameters (where we also see the Galois correspondence).  In the absolute case these are the same thing, but not in the ``relative" case.

 So let me discuss this second aspect of the Galois group, as well as the Galois theory.  We work with the notation so far; $P$ and$Q$ are $A$-definable, with $Q$ internal to $P$ and $Gal(Q/P,A)$ has its two incarnations as a definable group.  Remember our assumption that $T$ is totally transcendental, so let $M$ be the prime model over $A$.  (So all types over $A$ realized in $M$ are isolated.) Then using the $A$-definable function $f(-,-)$ we can find a fundamental system of solutions ${\bar b}$ inside $Q(M)$.  We will make the very strong assumption that
\newline
(*) $P(M)\subseteq dcl(A)$. (All elements of $P$ in $M$ are in the definable closure of $A$.) 

Let $G$ be the definable Galois group $Gal(Q/P,A)$ in either of its incarnations.  Let $B = dcl(A,b)$. 
As in ordinary Galois theory we let $Gal(B/A)$ be the group of elementary permutations of $B$ which fix $A$ pointwise.
\begin{Proposition} (i) $G(M)$ identifies with $Gal(B/A)$.
\newline
(ii) There is a Galois correspondence between definable (over $M$) subgroups of $G(M)$ and definably closed sets $C$ such that $A\subseteq C \subseteq B$.

\end{Proposition}

See \cite{Leon-Sanchez-Pillay} for a generalization of this, as well as proofs.

\vspace{5mm}
\noindent
Bruno Poizat was maybe the first to observe (in his seminal paper \cite{Poizat-imaginaries}), that the theory we have just described subsumes the Picard-Vessiot theory. Let me describe briefly the latter, to finish this Section of the article. 
Let us fix a differential field $(K, \partial)$ with algebraically closed field $C_{K}$ of constants (all in characteristic $0$). Let $\partial Y = AY$ be a linear differential equation over $K$ in vector form, namely $Y$ is a $n\times 1$ column vector of indeterminates and $A$ is an $n\times n$ matrix over $K$. 
An easy computation, plus differential algebra, yields a few basic facts:
\begin{Fact} (i) For any differential field $(L,\partial)\supseteq (K,\partial)$, the solution set $V(L)$ of $\partial Y = AY$ in $L$ is a vector space over $C_{L}$ of dimension $\leq n$.
\newline
(ii)  There is such a differential field $(L,\partial)$ such that $C_{L} = C_{K}$ and the $C_{L}$-dimension of $V(L)$ is precisely $n$.
\end{Fact}

In the context of Lemma 5.6(ii) let $Z$ be an $n\times n$ matrix  over $L$ whose columns form a basis of $V(L)$ over $C_{L} = C_{K}$. This is equivalent to asking that $Z$ be nonsingular and its columns are members of $V(L)$. 
The differential field $K(Z)$ generated over $K$ by the coordinates of $Z$ is what is called a Picard-Vessiot extension of $K$ for the equation $\partial Y = AY$. Let $Aut_{\partial}(K(Z)/K)$ be the group of automorphisms of the differential field $K(Z)$ which fix $K$ pointwise. For $\sigma\in Aut_{\partial}(K(Z)/K)$, let $c_{\sigma}\in GL_{n}(C_{K})$ be such that  $\sigma(Z) = Zc_{\sigma}$. Then the map taking $\sigma \to c_{\sigma}$ is an isomorphism of $Aut_{\partial}(K(Z)/K)$ with an algebraic subgroup $G$ of $GL_{n}(C_{K})$, and there is a Galois correspondence between algebraic subgroups of $G$ and intermediate differential fields between $K$ and $K(Z)$.

Poizat's approach is to work in $DCF_{0}$. Once one knows that there is some differential field extension $L$ of $K$ where $V(L)$ has dimension $n$ over $C_{L}$ then we can find such $Z$ inside the prime model $K^{diff}$ over $K$. Moreover the constants being strongly minimal and $C_{K}$ being algebraically closed it follows that $C_{K} = C_{K^{diff}}$. So taking $P$ to be the definable set ${\mathcal C}$ of constants and $Q$ to be the formula $\partial Y = AY$, and $A = K$,  Proposition 5.5 applies to get the above results. The subgroup $G$ of $GL_{n}(C_{L})$ in the  $PV$-theory is precisely the group of $K^{diff}$ points of the extrinsic Galois group (living in $P = {\cal C}$). 

Actually all this generalizes to Kochin's strongly normal theory concerning any definable sets internal to the constants, and Poizat's observations inspired a differential Galois theory properly extending the strongly normal theory (\cite{Pillay-DGTIV}), and relevant to new problems in functional transcendence (\cite{BP}).

\section{Topological dynamics}
Topological dynamics, in its abstract version,  is the study of topological  groups via their actions on compact (Hausdorff) spaces. Notions such as amenability, (universal) minimal flow, distal flow, and weak mixing, come from this subject.  Rather than give a series of definitions of the various notions, I will introduce them in the text, when and where relevant.  In general \cite{Glasner} will be our reference for topological dynamics.

Ludomir Newelski  \cite{Newelski} had the insights that (i) the fundamental theorem of stable group theory can be interpreted as a topological dynamical statement, and (ii) topological dynamics introduces notions and techniques which are useful in generalizing stable group theory outside the context of stable theories. 

We will focus in this section on definable groups, but mention briefly automorphism groups. 

Among the themes I will mention are the connections between the various quotients such as $G/G_{M}^{00}$ and topological dynamical invariants (which in hindsight is just a comparison of various kinds of compactifications) as well as consequences of definable amenability. 

By an action  of a topological group $G$ on a compact space $X$ we mean a group action $\alpha: G\times X \to X$ which is also continuous. When $G$ is discrete, the continuity condition just means that for each $g\in G$, $\alpha(g,-)$ is a homeomorphism of $X$.  We tend to restrict our attention here to the case where the group $G$ is discrete.  We often call the pair $(G,X)$ a $G$-flow.

What exactly is meant by Newelski's observations (i) and (ii)?  First we are considering a group $G$ (or $G(M)$) definable in a structure $M$, as a discrete group.  There is an obvious action of $G(M)$ on a certain compact space, namely on the space $S_{G}(M)$ of complete types over $M$ concentrating on $G$, where the action is by translation: $gp = tp(ga/M)$ where $a$ realizes $p$ (in $G({\bar M})$).  And this action is of course by homeomorphisms.  To deal with (i): when $Th(M)$ is stable, then the collection of generic types over $M$ (see Section 3.1) is the unique minimal closed invariant subset of $S_{G}(M)$, and has the structure of a compact (profinite) group, naturally isomorphic to $G/G^{0}$. 
This minimal $G(M)$-flow (and its quotients) is called an equicontinuous flow in the dynamics literature.
For (ii): in the not necessarily stable environment,  and given the action of $G(M)$ on $S_{G}(M)$, for example,  the study of  minimal subflows of this action (which exist) can give interesting model theoretic information about $G$, generalizing notions from stable group theory.  Actually we already have generalizations of stable group theory to simple theories and $NIP$ theories, and the second fits very well into the topological dynamics point of view (but not so much the first). 

The whole theory is really about so-called $G$-ambits, namely a group action $G$ on $X$ with a distinguished dense orbit, more precisely with a dinstinguished point $x_{0}\in X$ whose $G$-orbit is dense in $X$. 
In the case of $G(M)$ acting on $S_{G}(M)$ the dense orbit is just $G(M)$ itself (identified with the set of types in $S_{G}(M)$ realized in $M$), so $x_{0}$ is taken as the identity of $G$.

Important flows are minimal flows $(G,X)$ where {\em every}  orbit is dense.  We have already discussed above the Bohr compactification of a (discrete) group $G$, as a compact (Hausdorff) group $bG$ together with a homomorphism $G\to b(G)$ with dense image, and which is universal such.  Note that $b(G)$ is also a $G$-flow where the group action of $G$ on it is via left multiplication of the image of $G$ inside $bG$. Moreover $bG$ is minimal. 
In any case, $G$ will have a universal minimal flow (unique up to isomomorphism) which is again an important invariant of $G$.

An important aspect of (abstract) topological dynamics is the construction of the ``Ellis semigroup" attached to a group action.  This is a related group action of $G$ where the space on which the group acts has a multiplicative (semigroup) structure, continuous in the first coordinate.  The formal construction is as follows. We start with an action  of a (discrete) group $G$ on a space $X$. For each $g\in G$ we have the corresponding homeomorphism $\pi_{g}\in X^{X}$. The closure of $\{\pi_{g}:g\in G\}$ in $X^{X}$ (with the Tychonoff topology) is precisely $E(X)$, the Ellis semigroup of the action of $G$ on $X$.  $G$ also acts by homeomorphisms on the compact space $E(X)$, by composition of functions:  for $f\in E(X)$, and $g\in G$, $g(f) = \pi_{g}\circ f$. Moreover $E(X)$ is a $G$-ambit, where the orbit of the identity function is precisely $\{pi_{g}:g\in G\}$ and  is dense by definition.  Moreover $E(E(X))$ is naturally isomorphic to $E(X)$ as a $G$-flow. 
The semigroup structure $*$ on $E(X)$ is just composition of maps, and is continuous in the first coordinate.  Minimal subflows of $(G,E(X))$ coincide with minimal left ideals of the semigroup $E(X)$ and are basic invariants of the original flow $(G,X)$. 
Let ${\cal M}$ be such a minimal subflow (minimal closed $G$-invariant subset of $E(X)$) and let $u$ be an idempotent in ${\cal M}$. Then $u*{\cal M}$ is actually a subgroup of the semigroup $E(X)$, which we will here call the {\em Ellis group} attached to the original flow $(G,X)$,  which is unique up to isomorphism as an abstract group.  (This is not normal terminology in topological dynamics, but has become usual in model theory, following Newelski.)   In fact this Ellis group coincides (via acting on the right by $*$) with the group of automorphisms of the $G$-flow ${\cal M}$.  The so-called $\tau$-topology on $u*{\cal M}$ can be defined via this ``Galois action":  To each $g\in u*{\cal M}$, we have its graph $\Gamma(g) = \{(x, x*g):x\in {\cal M}\}\subseteq {\cal M}\times {\cal M}$. Then for a subset $A$ of $u*{\cal M}$ and $g\in u*{\cal M}$ we say that $g\in cl_{\tau}(A)$ if $\Gamma(g)$ is contained in the closure in ${\cal M}\times {\cal M}$ of the union of the $\Gamma(a)$ for $a\in A$.  

From  this definition of $cl_{\tau}(-)$ we obtain the $\tau$-topology on the ``Ellis group" $u*{\cal M}$. The topology is $T_{1}$ and (quasi) compact and the group operation $*$ is separately  continuous.  The intersection of the $\tau$-closures of $\tau$-neighbouhoods of the identity in $u*{\cal M}$ is a $\tau$-closed normal subgroup $H$ of $u*{\cal M}$, and the quotient 
$(u*{\cal M})/H$ is a compact Hausdorff topological group. 

What exactly is canonical about this construction which begins with an action of $G$ on $X$?  Well it turns out that there is a {\em universal} $G$-ambit, namely every other $G$-ambit is uniquely an image of the universal one.  Moreover the universal minimal flow  of $G$   corresponds to some/any minimal subflow of the universal ambit.  

If $X$ is the universal $G$-ambit then $X$ is naturally isomorphic to its Ellis semigroup, so is equipped with a semigroup structure $*$. 
Let ${\cal M}$ be a minimal subflow of $X$ and $u$ an idempotent. So the group $(u*{\cal M}, *)$ is an invariant of $G$, even as a (not necessarily Hausdorff) topological group,  and I will call it the Ellis group of $G$ (which is not usual terminology in dynamics). Moreover the compact Hausdorff group $(u*{\cal M})/H$ is called the {\em generalized Bohr compactification} of $G$.  The terminology is due to Glasner although the group seems to have been defined originally by Ellis.

The point is that for groups such as $SL(2,\R)$  (as a discrete or topological group), the Bohr compactification is trivial, but not the generalized Bohr compactification.

Other notions attached to the topological  dynamical study of (say discrete) groups are {\em amenability}: every $G$-flow $X$ supports a $G$-invariant Borel probability measure. Extreme amenability of $G$ means that every such action has a fixed point (so a Dirac, as a measure). 

\vspace{5mm}
\noindent
What is the connection with model theory?  Well given a group $G$ definable in/over a model $M$ we have the action of $G(M)$ on the space $S_{G}(M)$, and we can simply use the machinery above to get information about this action, such as passing to the Ellis semigroup $E(S_{G}(M))$ etc. 

On the other hand, we could try to formulate analogues of notions around actions of discrete (even nondiscrete) groups on compact spaces but in a suitable ``definable" category. This has been done or at least attempted in several papers, although not always with consistent notation. 
One interesting point with many of these model-theoretic analogues is that taking  set theory as our theory $T$ i.e. the theory of the standard model $({\mathbb V},\epsilon)$),  and viewing group $G$ as a group definable in this standard model,  we recover the topological dynamics of the discrete group $G$. 

As an example, take a discrete group $G$ as a group definable in $M = ({\mathbb V}, \epsilon)$. The universal $G$-ambit is well known to be  $\beta G$, also known as the Stone-Cech compactification of $G$, explicitly the space of ultrafilters on the Boolean algebra of all subsets of $G$. But this is precisely the space of types $S_{G}(M)$.

We now revert to the general model theoretic context of a complete theory $T$, monster model ${\bar M}$ and $\emptyset$-definable group $G$. Unless stated otherwise we consider $G(M)$ as a discrete group.
\begin{Proposition} (i) Consider the flow $(G(M), S_{G}(M))$. Let $N$ be an $|M|^{+}$-saturated model containing $M$, then the Ellis semigroup of the action is an image, as both a $G(M)$-flow and semigroup of  the space $S_{G,M}(N)$ of types in $S_{G}(N)$ which are finitely satisfiable in $M$ (with the natural action of $G(M)$). 
\newline
(ii) If all types over $M$ are definable, then $S_{G}(M)$ equals $S_{G,M}(N)$, and $S_{G}(M) = E(S_{G}(M))$. 
\end{Proposition}
\begin{proof} (i)  The main point is that  $f:S_{G}(M) \to S_{G}(M)$ is in the closure of the set of functions $\pi_{g}$ for $g\in G(M)$ (where $\pi(g)(p)$ is precisely $gp$) if and only if there is some $p'\in S_{G,M}(N)$ such that for any $p\in S_{G}(M)$, some/any realization $a$ of $p$ in $G(N)$, and some/any realization $g'$ of $p'$ (in $G({\bar M})$, $f(p) = tp(g'a/M)$.
This is basically the connection between  closures of families of functions and  finitely satisfiable types, and is an important aspect of the connection between function theory (or functional analysis) and model theory.  The rest is left to the reader, although it is worth describing the semigroup structure on $S_{G,M}(N)$.  Let $p,q\in S_{G,M}(N)$. Let $b$ realize $q$. Then $p$ has a unique extension  $p'$ over $N,b$ which is finitely satisfiable in $M$. Let $a$ realize $p'$. Then $p*q = tp(ab/N)$, which  {\em is} finitely satisfiable in $M$, and depends only on $p$ and $q$. 
\newline
(ii) If all types over $M$ are definable, then every complete type $p$ over $M$ has a unique extension $p'$  over $N$ which is finitely satisfiable in $M$.  This gives a homeomorphism between $S_{G}(M)$ and $S_{G,M}(N)$.  In fact much less than the definability of all types over $M$ is required to get the conclusion here, but never mind. 
\end{proof}

Actually the space $S_{G,M}(N)$ can be accessed directly as follows:
\begin{Definition}  A subset of $G(M)$ is said to be {\em externally definable} if it is of the form $X\cap G(M)$ where $X$ is a definable, maybe with parameters, subset of $G = G({\bar M})$.
\end{Definition} 
The collection of  externally definable subsets of $G(M)$ is clearly a Boolean algebra, and we let $S_{G,ext}(M)$ denote its space of ultrafilters (on which $G(M)$ acts). 
\begin{Lemma}  $S_{G,ext}(M)$ is isomorphic, as a $G(M)$-flow, to $S_{G,M}(N)$. 
\end{Lemma}

If all types over $M$ are definable, then of course externally definable subsets of $G(M)$ coincide with definable subsets, and
everything fits together. 

\begin{Definition} (i) A map $f$ from $G(M)$ to a compact space $X$ is said to be definable  if it factors through a continuous function from $S_{G}(M)$ to $X$.
Likewise $f$ is to be externally definable if it factors through a continuous function from $S_{G,ext}(M)$ to $X$.
\newline 
(ii) An action of $G(M)$ on a compact set $X$ (by homeomorphisms) is said to be definable if for each $x\in X$, the map taking $g\in G(M)$ to $gx\in X$ is a definable map from $G$ to $X$.
Likewise for the notion of the action being externally definable. 
\end{Definition}

Let me remark that the notion of a ``definable action" of a definable group on a compact space is supposed to be some kind of analogue of a (jointly) continuous action of a topological group on a compact space. But maybe we did not yet find the correct analogue.  Anyway with these notions we obtain some kind of canonicity or universality of the action of $G(M)$ on the various type spaces:
\begin{Proposition} (i) The action of $G(M)$ on $S_{G,ext}(M)$ is externally definable, and moreover with $x_{0}$ the identity of $G(M)$, $(S_{G,ext}(M), x_{0})$ is the universal externally definable $G(M)$-ambit.
\newline
(ii) Likewise, if all types over $M$ are definable, $(S_{G}(M), x_{0})$ is the universal definable $G(M)$-ambit. 
\end{Proposition} 

Hence, for various reasons, if we consider the group $G(M)$, and $X = S_{G}(M)$, then the Ellis group $u*{\cal M}$ and quotient  $(u*{\cal M})/H$ attached to $X$ (i.e. constructed from the semigroup $E(X) = S_{G,ext}(M)$) can be considered fairly canonical, and there is no harm calling them the definable Ellis group and definable generalized Bohr compactification (although using ``externally definable" in place of definable, would also make sense). 

\subsection{Newelski conjecture}
In the paper \cite{Newelski} already mentioned above, Newelski asked some questions at the end, including about what he calls weak generic types (which I have not defined) and clarifying relationships between groups such as $G/G^{00}$ and Ellis groups. Such as, are they equal assuming a bounded number of weak generics?
This led to some people such as myself dubbing as ``Newelski's conjecture" that in certain tame contexts, such as $NIP$ or $o$-minimal, given a group $G$ definable over a structure $M$, then $G/G_{M}^{00}$ equals the Ellis group $u^*{\cal M}$ described above. Newelski noted already that there is a natural surjective homomorphism from this Ellis group to $G/G^{00}_{M}$.  So the question was whether this must be an isomorphism.

The first result on the question was negative, in the sense of giving a counterexample for groups in $o$-minimal structures without any further hypotheses. See \cite{GPP-SL(2)}. See also \cite{PPY} for a generalization to $SL(2,\Q_{p})$. 
\begin{Proposition} Let $M$ be the structure $(\R, +, \times)$ and $G = SL_{2}$, so $G(M) = SL_{2}(\R)$. In this case all types over $M$ are definable. Then the Ellis group is nontrivial (namely the $2$ element group),  whereas $G/G^{00}$ is trivial.
\end{Proposition}

So the rather loose ``Newelski conjecture" was tightened to ask whether $G/G^{00}$ equals the Ellis group, when $T$ is $NIP$ and $G$ is definably amenable.  (Actually the assumptions are related to Newelski's condition of a bounded number of weak generic types).  Anyway this was proved in \cite{Chernikov-Simon}.

\begin{Proposition} Assume $T$ is $NIP$ and $G$ is definably amenable. Then the natural homomorphism from the Ellis group to $G/G^{00}$ is an isomorphism.
\end{Proposition}

The paper \cite{Chernikov-Simon} also includes a classification of the invariant Keisler measures on definably amenable groups in $NIP$ theories. 

A possible refinement of the Newelski conjecture, which takes account of the counterexample, is that when $T$ is $NIP$ and $G$ is a definable group then the Ellis group $u*{\cal M}$ of $G(M)$ coincides with the definable generalized Bohr compactification $(u*{\cal M})/H$ which just means that the $\tau$-topology on $u*{\cal M}$ is already Hausdorff. 
This is  true in the definably amenable case, and the isomorphism of the Ellis group with $G/G^{00}_{M}$ is actually an isomorphism of topological groups where the Ellis group is equipped with its $\tau$-topology (although it is not part of the statement in \cite{Chernikov-Simon}), which follows from the discussion in the next subsection. 

\subsection{$G/G_{M}^{000}$}
As mentioned earlier the group $G/G_{M}^{000}$ belongs to model theory, although it is not a compact, Hausdorff group, and in fact  its status as a mathematical object has always been a bit problematic.  For $g\in G = G({\bar M})$ the class of $g$ modulo $G_{M}^{000}$ depends only on $tp(g/M)$, so $G/G_{M}^{000}$ is  a ``quotient" of $S_{G}(M)$ by what turns out to be an $F_{\sigma}$ equivalence relation (on the compact  space $S_{G}(M)$). When $T$ and $M$ are countable, we are in the environment of Borel equivalence relations on Polish spaces, and it turned out to be profitable to study the complexity of the relevant equivalence relation.  The general conjecture was ``smooth" implies that $G_{M}^{000} = G_{M}^{00}$ and so the quotient is compact, Hausdorff. (The Lascar group analogue was mentioned earlier.)

 However, already looking at examples, one had the feeling that more structure could be given to $G/G_{M}^{000}$, in fact a structure belonging to classical noncommutative geometry.   This turns out to be true using the topological dynamics machinery. See \cite{Krupinski-Pillay}.
\begin{Proposition} The natural map $f$ from $S_{G,ext}(M)$ to $G/G_{M}^{000}$ induces a surjective homomorphism $\bar f$ from $(u*{\cal M})/H$ (the definable generalized Bohr compactification) to $G/G_{M}^{000}$. 
\end{Proposition}

We also have a surjective homomorphism $\pi$ say from $G/G_{M}^{000}$ to the definable Bohr compactification $G/G_{M}^{00}$, and let $K$ be the kernel of $\pi$ which identifies with $G_{M}^{00}/G_{M}^{000}$.
\begin{Proposition}  ${\bar f}^{-1}(K)$ is a closed subgroup of  $(u*{\cal M}/H)$, and the kernel of the surjective homomorphism from ${\bar f}^{-1}(K)$ to $K$ given by ${\bar f}$ is a dense subgroup.
 Hence  $K$ has canonically the structure of the quotient of a compact Hausdorff group by a dense subgroup. 
\end{Proposition}

\subsection{Amenability and automorphism groups}
The study of automorphism groups (as topological groups) from the point of view of topological dynamics has been very profitable, and in \cite{KPT} the connection was made with structural Ramsey theory (for automorphism groups of countable structures). 
When the theory $Th(M)$ of the countable structure $M$ is $\omega$-categorical, then $Aut(M)$, as a topological group, is an invariant of  $Th(M)$ (in fact $Th(M)$ can be recovered from $Aut(M)$), so the topological dynamics of $Aut(M)$ is something also model-theoretic. 

In \cite{Krupinski-Pillay2} we proved that if $G$ is a definable group which is definably amenable, then $G_{M}^{000} = G_{M}^{00}$, using elaborations of the stabilizer theorem for approximate subgroups (and \cite{Massicot-Wagner}). This was adapted in the same paper to the automorphism group context to show that if $T$ is $\omega$-categorical and the automorphism group of the (unique) countable model of $T$ is amenable, then $T$ is $G$-compact, i.e. $Gal_{L}(T) = Gal_{KP}(T)$. 
In \cite{HKP} we introduced the notion of an amenable first order theory $T$ (which is considerably weaker than $Aut(M)$ being amenable as a topological group when $T$ is $\omega$-categorical) and proved:
\begin{Proposition} If $T$ is amenable as a first order theory, then $T$ is $G$-compact. 
\end{Proposition}

\section{The free group}

An important question of Tarski was whether all free groups $F_{n}$ on $n$ generators for $n\geq 2$ are elementarily equivalent. 
This was solved by two independent groups of researchers, Kharlampovich-Myasnikov and Sela, in a series of papers. (See \cite{KM}, \cite{Sela} for example, also for more references.)  So we have a complete theory $T_{fg}$ in the language of groups (with symbols for the group operation, inversion, and the identity element) which is the common theory of the $F_{n}$ ($n\geq 2$).

Subsequently Sela \cite{Sela-stability} proved that $T_{fg}$ is stable. This result is accepted by logicians and model theorists, although there has as yet been no independent proof (by logicians). One would like an account based on the combinatorics of words in the alphabet or language mentioned above. Possibly this is too naive a set-up, compared with the algebraic-topological (geometric group-theoretic) framework of Sela.  Sela's solutions, both to Tarski's problem and stability involved proving a relative quantifier elimination theorem down to Boolean combinations of $\forall\exists$ formulas.

Nevertheless, we can choose to take the stability of $T_{fg}$ for granted, and see where it leads us, in terms of problems and questions.
$T_{fg}$ is a {\em new stable group} (like a ``new leaf" in the film with Walter Matthau and Elaine May, or like a ``new galaxy" as in no film that I know), with very different properties from the already known stable groups. The subject ``algebraic geometry over the free group" is about the structure of solutions of equations in free groups, but in fact maybe should be about the category $Def(T_{fg})$.

Sela's result that the natural inclusions of the $F_{n}$ for $n\geq 2$ are elementary, implies that the free groups on infinitely many generators are models of $T_{fg}$. It is natural to ask whether (parts of) the proof of stability of $T_{fg}$ can be simplified by looking directly at  the models $F_{\kappa}$ for $\kappa$ infinite.  There are a couple of examples, noted by Poizat in his early but fundamental paper on stable groups  \cite{Poizat-generics}. The set $I$ of free generators of $F_{\omega}$ is an indiscernible set in $F_{\omega}$, and for any definable (in fact invariant over a finite set) subset $X$ of $F_{\omega}$, $X$ intersects $I$ in a finite or cofinite set.  It follows that:
\begin{Fact} The free group is connected (no proper definable subgroup of finite index) and moreover the type of some/any element of $I$ over $\emptyset$ is the unique generic type of $T_{fg}$ over $\emptyset$, which we call $p_{0}$.  
\end{Fact} 

In the free group elements have unique square roots. If $a$ realizes $p_{0}$ then $a$ is not a square but $a$ is algebraic over $a^{2}$, from which Poizat concludes:
\begin{Fact} $T_{fg}$ is not superstable.
\end{Fact} 

We did not define regular types in this article. Regular types are generalizations of types of Morley rank $1$ or $U$-rank $1$.  Nevertheless the same argument shows that the $p_{0}$ is not regular.
Some more technical results about $p_{0}$ were proved in \cite{Pillay-weight},  and \cite{Sklinos-generic}, including that $p_{0}$ has infinite weight, even witnessed in $F_{\omega}$: for $a$ realizing $p_{0}$ there is some infinite $f$-independent sequence $(b_{i}:i<\omega)$ such that $a$ forks with $b_{i}$ for all $i$.

We still  know very little about the category $T_{fg}$, although Sklinos and collaborators have made some progress (but using heavily the set-up and context of Sela, where not everything is accessible to model-theorists).  
Quite possibly a firmer understanding of why $T_{fg}$ is stable would make a difference. 
As a related case where our knowledge is basically complete, let us consider the theory of $F_{1}$, otherwise known as 
$T = Th(\Z,+)$.  Now $T$ is not totally transcendental (as there is an infinite descending chain $2^{n}\Z$ of subgroups), but it is superstable, even of rank $1$, and it is the archetypical example of a superstable, non $t.t$ theory, but where both categories $Def(T)$ and $Mod(T)$ are well-understood.  One of the main things is that as an abelian  group (with no additional structure), $(\Z. +)$ is a $1$-based group, so the definable sets (in arbitrary models) are Boolean combinations of cosets of $acl(\emptyset)$-definable subgroups of $\Z^{n}$. These can be easily computed. And also any model of $T$ is of the form $G \oplus \Q^{(\kappa)}$ where $G$ is an elementary substructure of the profinite completion ${\hat\Z}$ of $\Z$. In particular (assuming $CH$), the $\aleph_{1}$-saturated models of $T$ can be described as ${\hat \Z} \oplus \Q^{\kappa}$ (for $\kappa$ a cardinal $\geq \aleph_{1}$).  

So this begs the question of whether  there is a classification of the $\aleph_{1}$-saturated models of $T_{fg}$. 
Now a nonstructure property for $\aleph_{1}$-saturated models of a stable (possibe unsuperstable) theory $T$ is that $T$ has the so-called $DOP$ (dimensional order property of Shelah).  See \cite{Pillay-Sklinos-DOP} for both the definitions and the proof that $T_{fg}$ has the $DOP$. So there is NO reasonable structure theorem for $\aleph_{1}$-saturated models of $T_{fg}$. This may explain why it has so far been difficult to describe ``nonstandard" models of $T_{fg}$ (where nonstandard means not free), in particular saturated models. 

Let now $F$ denote an arbitrary model of $T_{fg}$, such as a ``standard model" ($F_{n}$ for $n\geq 2$).
Some obvious definable subgroups of $F$ are the centralizers of elements $a$, $C_{F}(a)$.  Some important observations are that these definable subgroups, with the  induced structure from $F$, are just models of $Th(\Z,+)$ (maybe with constants added).   So these are examples of $1$-based $U$-rank $1$ groups definable/interpretable in $T_{fg}$.

Here are some questions.
\begin{Question}
(i) What are the $U$-rank $1$ formulas/ types,  in $T_{fg}$. Are they only the centralizers described earlier?
\newline 
(ii) What are the superstable formulas (definable sets) in $T_{fg}$?
\newline
(iii) What are the definable/interpretable fields in $T_{fg}$?
\newline
(iv) What are the definable/interpretable groups in $T_{fg}$?
\newline
(v) What are the regular types in $T_{fg}$? 
\end{Question}

For example in \cite{Byron-Sklinos} it is proved that there are no definable infinite fields.

\end{document}